\nonstopmode
\documentclass[10pt]{amsart}
\usepackage{graphicx}
\usepackage{latexsym}
\usepackage{fancyhdr}
\usepackage{amsmath, amssymb, amsthm}
\usepackage[all]{xy}
\usepackage{pdflscape}
\usepackage{longtable}
\usepackage{rotating}
\usepackage{verbatim}
\usepackage{hyperref}
\hypersetup{
    linkcolor = blue,
    citecolor = blue,
    urlcolor  = blue,
    colorlinks = true,
}

\usepackage{cleveref}
\usepackage{subfigure}
\usepackage{mathrsfs}
\usepackage{mdwlist}
\usepackage{dsfont}
\usepackage{mathtools}
\usepackage{float}
\usepackage{color}
\usepackage{stmaryrd}

\usepackage{tkz-base}
\usepackage{pgfplots}
\pgfplotsset{compat=newest}

\usepackage{tkz-euclide}
\usepackage{etoolbox}

\definecolor{teal}{rgb}{0.0, 0.5, 0.5}

\newcounter{mnotecount}[section]

\newcommand{\rmnote}[1]{}

\allowdisplaybreaks

\DeclareFontFamily{U}{mathb}{\hyphenchar\font45}
\DeclareFontShape{U}{mathb}{m}{n}{
      <5> <6> <7> <8> <9> <10> gen * mathb
      <10.95> mathb10 <12> <14.4> <17.28> <20.74> <24.88> mathb12
      }{}
\DeclareSymbolFont{mathb}{U}{mathb}{m}{n}
\DeclareFontSubstitution{U}{mathb}{m}{n}

\let\dot\relax
\DeclareMathAccent{\dot}{0}{mathb}{"39}
\let\ddot\relax
\DeclareMathAccent{\ddot}{0}{mathb}{"3A}
\let\dddot\relax
\DeclareMathAccent{\dddot}{0}{mathb}{"3B}
\let\ddddot\relax
\DeclareMathAccent{\ddddot}{0}{mathb}{"3C}

\theoremstyle{plain}
\newtheorem*{theorem*}{Theorem}
\newtheorem{theorem}{Theorem}[section]
\newtheorem*{lemma*}{Lemma}
\newtheorem{lemma}[theorem]{Lemma}
\newtheorem*{assumption*}{Assumption}

\newtheorem*{proposition*}{Proposition}
\newtheorem{proposition}[theorem]{Proposition}
\newtheorem*{corollary*}{Corollary}
\newtheorem{corollary}[theorem]{Corollary}
\newtheorem*{claim*}{Claim}

\newtheorem*{conjecture*}{Conjecture}

\newtheorem*{question*}{Question}
\newtheorem*{result*}{Result}

\theoremstyle{definition}
\newtheorem*{definition*}{Definition}
\newtheorem{definition}[theorem]{Definition}
\newtheorem*{example*}{Example}

\newtheorem*{algorithm*}{Algorithm}
\newtheorem*{remark*}{Remark}
\newtheorem*{remarks*}{Remarks}
\newtheorem{remark}[theorem]{Remark}

\newtheorem*{convention*}{Convention}

\newenvironment{hypo}[1]
{\par\smallskip\noindent\textbf{#1}\begin{itshape}}
{\end{itshape}\par\smallskip}

\numberwithin{equation}{section}

\def\al{\alpha}
\def\be{\beta}
\def\ga{\gamma}
\def\de{\delta}
\def\ep{\epsilon}

\def\ka{\kappa}
\def\la{\lambda}

\def\rh{\rho}

\def\si{\sigma}

\def\ta{\tau}

\def\vh{\varphi}
\def\ch{\chi}
\def\ps{\psi}
\def\om{\omega}

\def\Ga{\Gamma}
\def\De{\Delta}

\def\La{\Lambda}

\def\N{\mathbb{N}}

\def\R{\mathbb{R}}

\def\cP{\mathcal{P}}

\def\cT{\mathcal{T}}

\def\sD{\mathscr{D}}

\def\sS{\mathscr{S}}

\def\p{\partial}

\def\<{\langle}
\def\>{\rangle}
\renewcommand{\o}{\circ}

\def\supp{\on{supp}}

\def\ol{\overline}

\let\on=\operatorname

\newcommand{\sr}[1]%
{\ifmmode{}^\dagger\else${}^\dagger$\fi\ifvmode
\vbox to 0pt{\vss
 \hbox to 0pt{\hskip\hsize\hskip1em
 \vbox{\hsize3cm\raggedright\pretolerance10000
 \noindent #1\hfill}\hss}\vss}\else
 \vadjust{\vbox to0pt{\vss%
 \hbox to 0pt{\hskip\hsize\hskip1em%
 \vbox{\hsize3cm\raggedright\pretolerance10000%
 \noindent #1\hfill}\hss}\vss}}\fi%
}

\providecommand{\mapsfrom}{\kern.2em%
\setbox0=\hbox{$\leftarrow$\kern-.10em\rule[0.26mm]{0.1mm}{1.3mm}}\box0%
\kern.3em}

\title[Uniform extension of definable $C^{m,\om}$-Whitney jets]
{Uniform extension of definable $C^{m,\om}$-Whitney jets}

\author[Adam Parusi\'nski and  Armin Rainer]
{Adam Parusi\'nski and Armin Rainer}

\address {Adam Parusi\'nski: Universit\'e C\^ote d'Azur,  CNRS,  LJAD, UMR 7351, 06108 Nice, France}
\email{adam.parusinski@univ-cotedazur.fr}

\address{Armin Rainer: Fakult\"at f\"ur Mathematik, Universit\"at Wien,
Oskar-Morgenstern-Platz~1, A-1090 Wien, Austria}
\email{armin.rainer@univie.ac.at}

\begin{document}

\begin{abstract}
    We show that definable Whitney jets of class $C^{m,\om}$, 
    where $m$ is a nonnegative integer and $\om$ is a modulus of continuity,
    are the restrictions of definable $C^{m,\om}$-functions; ``definable'' refers to 
    an arbitrary given o-minimal expansion of the real field. 
    This is true in a uniform way: any definable bounded family of Whitney jets of class $C^{m,\om}$ 
    extends to a definable bounded family of $C^{m,\om}$-functions.
    We also discuss a uniform $C^m$-version and 
    how the extension depends on the modulus of continuity.
\end{abstract}

\thanks{This research was funded in whole or in part by the Austrian Science Fund (FWF) DOI 10.55776/P32905.
For open access purposes, the author has applied a CC BY public copyright license to any author-accepted manuscript version arising from this submission.
Additionally, the research was supported by Oberwolfach Research Fellows (OWRF) ID 2244p.}
\date{\today}

\maketitle


\section{Introduction}

Let an o-minimal expansion of the real field be fixed. 
Throughout the paper, a set $X \subseteq \R^n$ is called \emph{definable} if it is definable in this fixed o-minimal structure. 
A map $\vh : X \to \R^m$ is definable if its graph $\Ga(\vh):=\{(x,\vh(x)) : x \in X\}$ is a definable subset of $\R^n \times \R^m \cong \R^{n+m}$ (this natural identification is used throughout the paper). 
We assume familiarity with the basics of o-minimal structures; see \ \cite{vandenDriesMiller96} and \cite{vandenDries98}.

Due to Kurdyka and Paw{\l}ucki \cite{Kurdyka:1997ab,Kurdyka:2014aa} and 
Thamrongthanyalak \cite{Thamrongthanyalak:2017aa}
we have the definable $C^m$ Whitney extension theorem:

\begin{theorem} \label{thm:Cm}
    Let $0\le m\le p$ be integers. Let $E\subseteq \R^n$ be a definable closed set. 
    Any definable Whitney jet of class $C^m$ on $E$ extends to a definable $C^m$-function on $\R^n$
    which is of class $C^p$ outside $E$.  
\end{theorem}

We prove a $C^{m,\om}$-version of this result.

\begin{theorem} \label{thm:main}
    Let $0 \le m\le p$ be integers. Let $\om$ be a modulus of continuity.
    Let $E\subseteq \R^n$ be a definable closed set.
    Any definable Whitney jet of class $C^{m,\om}$ on $E$ extends to a definable $C^{m,\om}$-function on $\R^n$ 
    which is of class $C^p$ outside $E$.
\end{theorem}

By a \emph{modulus of continuity} we always mean a positive, continuous, increasing, 
and concave function $\om : (0,\infty) \to (0,\infty)$ 
such that $\om(t) \to 0$ as $t \to 0$.
We say that \emph{$\om$ is a modulus of continuity for a function $f : S \to \R$}, defined on a subset $S \subseteq \R^n$,
if there exists a constant $C>0$ such that  
\begin{equation} \label{eq:omHoelder}
    |f(x)-f(y)| \le C \, \om(|x-y|) \quad \text{ for all } x,y \in S.
\end{equation}
The class $C^{m,\om}$ consists of $C^m$-functions that are globally bounded together with its partial derivatives up to order $m$ 
and whose partial derivatives of order $m$ satisfy a global $\om$-H\"older condition of the type \eqref{eq:omHoelder}.
See \Cref{sec:jets}.

We use \Cref{thm:main} in \cite{ParusinskiRainer:2023ab} to show that a definable function 
$f : E \to \R$ on a definable closed 
set $E \subseteq \R^n$ that has a $C^{1,\om}$-extension to $\R^n$ also has a definable $C^{1,\om}$-extension. 
(In \cite{ParusinskiRainer:2023ab} we assume that $\om$ is definable, but not in the present paper.)
In fact, this application was one of our main motivations for proving \Cref{thm:main}.

We will show that 
the definable extension of Whitney jets of class $C^{m,\om}$ can be done in a bounded way:

\begin{theorem} \label{thm:mainuniform}
    Let $0 \le m\le p$ be integers. Let $\om$ be a modulus of continuity.
    Let $(E_a)_{a \in A}$ be a definable family of closed subsets of $\R^n$.
    For any definable bounded family $(F_a)_{a \in A}$ of Whitney jets of class $C^{m,\om}$ on $(E_a)_{a \in A}$ 
    there exists a definable bounded family $(f_a)_{a \in A}$ of $C^{m,\om}$-extensions to $\R^n$ of $(F_a)_{a \in A}$
    such that $f_a$ is of class $C^p$ outside $E_a$ for all $a \in A$.
\end{theorem}

Clearly, boundedness is understood with respect to the natural norms; see \Cref{sec:jets} for precise definitions.
\Cref{thm:main} follows as a special case from \Cref{thm:mainuniform}.
And already the case that $(F_a)_{a \in A}$ is a definable bounded family of Whitney jets of class $C^{m,\om}$ 
on a \emph{fixed set} $E=E_a$, for all $a \in A$, is very interesting.
However, the method of proof (by induction on dimension) 
necessitates to consider the general case that the families of Whitney jets are defined on 
definable families of sets $(E_a)_{a \in A}$.

The construction of the extension in \Cref{thm:mainuniform} depends on $\om$ only in a weak sense. 
We may, for instance, let the modulus of continuity $\om_a$ depend as well on $a \in A$ if we impose 
that there is a constant $C>0$ such that $C^{-1} \le \om_a(1) \le C$ for all $a \in A$.
This will be discussed in detail in \Cref{sec:omdep} in which we present a more general version of \Cref{thm:mainuniform}. 
As a consequence, we deduce from \Cref{thm:mainuniform} a uniform version of the $C^m$-result \Cref{thm:Cm} on compact sets:

\begin{theorem} \label{thm:Cmbd}
    Let $0 \le m\le p$ be integers. 
    Let $(E_a)_{a \in A}$ be a definable family of compact subsets of $\R^n$.
    For any definable bounded family $(F_a)_{a \in A}$ of Whitney jets of class $C^{m}$ on $(E_a)_{a \in A}$ 
    there exists a definable bounded family $(f_a)_{a \in A}$ of $C^{m}$-extensions to $\R^n$ of $(F_a)_{a \in A}$
    such that $f_a$ is of class $C^p$ outside $E_a$ for all $a \in A$.
\end{theorem}

\Cref{thm:Cmbd} is proved in \Cref{sec:proofCm}. 
We deduce a local version of \Cref{thm:mainuniform} in \Cref{ssec:local}
and apply \Cref{thm:mainuniform} in \Cref{ssec:corrLip} to get a definable version of a correspondence, 
due to Shvartsman \cite{Shvartsman:2008aa},
between Whitney jets of class $C^{m,\om}$ and certain Lipschitz maps.

The proof of \Cref{thm:mainuniform} (which builds upon the one of \Cref{thm:Cm} devised in 
\cite{Kurdyka:1997ab,Kurdyka:2014aa,Thamrongthanyalak:2017aa} and 
also Paw{\l}ucki \cite{Pawlucki08aa} and is very different from Whitney's classical method 
\cite{Whitney34a}) rests on two main cornerstones:
\begin{enumerate}
    \item \emph{Two versions of Gromov's inequality \cite{Gromov87}; one classical, the other   
        incorporating the modulus of continuity.} These are inequalities for the derivatives of a definable function.
        Since the constants that appear in them are universal, it is not difficult to get them uniform for definable families of functions.
        See \Cref{cor:La1top2} and \Cref{prop:uniformGromov4}.
    \item \emph{Uniform $\La_p$-stratification of definable families of sets.} 
        Roughly speaking, definable families of sets admit a stratification into a finite number of 
        cells that are defined by functions satisfying certain estimates (for their derivatives up to order $p$). 
        The constants in these estimates and the number of cells are independent of the 
        parameter of the family. See \Cref{thm:uniformLapstrat}.
        This is essential for the uniform extension theorem \ref{thm:mainuniform}.
        We think that it is also of independent interest.
\end{enumerate}

It is a natural question if there exists even a continuous and/or linear extension operator 
for definable Whitney jets of class $C^{m,\om}$ (or $C^m$) on a definable closed set $E \subseteq \R^n$.
This remains an open problem.
The theorem of Bartle and Graves \cite{BartleGraves52} (see also \cite[Theorem 1.6]{BrudnyiBrudnyi12Vol1}) 
is not applicable since the normed spaces of definable jets and functions (defined in \Cref{sec:jets}) are not complete.

Azagra, Le Gruyer, and Mudarra \cite{Azagra:2018aa} 
give an explicit formula for the extension of Whitney jets of class $C^{1,1}$ 
with an optimal control of the norms; for definable input this explicit formula yields a definable 
$C^{1,1}$-extension.  
See \cite[Sections 4.2--4]{ParusinskiRainer:2023ab}.

Let us point out that Paw{\l}ucki \cite{Pawlucki08aa} presents a continuous linear extension operator 
for (not necessarily definable) Whitney jets on definable closed sets which preserves (up to a multiplicative constant) 
the modulus of continuity. This extension operator is a finite composite of operators that 
preserve definability on the one hand or are defined by integration with respect to a parameter 
(more precisely, convolution)
on the other hand; in general, the latter leads out of the original o-minimal structure.

While \cite{Pawlucki08aa} was a important source of inspiration for handling the modulus of continuity,
the main difficulty (besides getting everything uniformly bounded) was to replace the convolution operators by definable operations 
which at the same time allow for preserving the modulus of continuity.

The paper is organized as follows.
In \Cref{sec:strat} the main geometric tools are prepared: Gromov's inequality and uniform $\La_p$-stratification. 
We present in \Cref{sec:jets} background on definable bounded families of Whitney jets of class $C^{m,\om}$, 
most notably, how they behave under pullback along a definable family of $\La_p$-regular maps.
The proof of \Cref{thm:mainuniform} is carried out in \Cref{sec:proof}.
In the final \Cref{sec:applications} 
we give the mentioned applications, discuss dependence on the modulus of continuity, and prove \Cref{thm:Cmbd}.

\subsection*{Notation}
Let $\N=\{0,1,2,\ldots\}$ be the set of nonnegative integers.
We denote by $d(x,S) := \inf_{y \in S} | x -y|$ the Euclidean distance in $\R^n$ of a point $x$ to a subset $S$ of $\R^n$, 
with the convention $d(x,\emptyset) := +\infty$. 
The open Euclidean ball with center $x \in \R^n$ and radius $r>0$ is denoted by $B(x,r) := \{y \in \R^n : |x-y|<r\}$.
The closure of a set $S$ is denoted by $\ol S$ and the frontier of $S$ by $\p S:= \ol S \setminus S$. 
If $S$ is a subset of $\R^k$, we write $S \times 0$ for the set $\{(u,w) \in \R^k \times \R^\ell : u \in \R^k, w = 0\}$.
The graph of a map $\vh$ is denoted by $\Ga(\vh)$.
For real valued nonnegative functions $f,g$ we write $f \lesssim g$ if $f \le C g$ for some universal constant $C>0$.
In particular, it should be always understood that $C$ is independent of $a \in A$, i.e., the parameter  
which we consistently use in parameterized families of sets and maps.  
We write $f \approx g$ if $f \lesssim g$ and $g \lesssim f$.
We use standard multi-index notation and in this context $(i) \in \N^n$ is the multi-index 
$(0,\ldots,0,1,0,\ldots,0)$ with $1$ in the $i$-th entry.

\section{Uniform $\La_p$-stratifications} \label{sec:strat}

The existence of uniform $\La_p$-stratifications (\Cref{thm:uniformLapstrat}) is
based on an inequality of Gromov \cite{Gromov87}, of which we need two versions, 
and on uniform $L$-regular decomposition due to Kurdyka and Parusi\'nski \cite{KurdykaParusinski06}.

\subsection{Definable families of sets and maps}

Let $A$ be a definable subset of $\R^N$. 
A family $(E_a)_{a \in A}$ of definable sets $E_a \subseteq \R^n$ 
is called a \emph{definable family} if the associated set
\begin{equation} \label{eq:associatedset}
    E := \bigcup_{a \in A} \{a\} \times E_a
\end{equation}
is a definable subset of $\R^N \times \R^n$.
Conversely, any definable subset $E \subseteq \R^N \times \R^n$ defines 
a definable family $(E_a)_{a \in A}$ by setting 
$A := \{a \in \R^N : \exists x \in \R^n, (a,x) \in E \}$
and $E_a := \{x \in \R^n : (a,x) \in E\}$, where $a \in A$.
If we allow $E_a = \emptyset$, we may just take $A = \R^N$.

A family $(E_a')_{a \in A}$ of subsets $E_a' \subseteq E_a$ is said to be a 
\emph{definable subfamily} of $(E_a)_{a \in A}$ if the associated set $E'$ (defined in analogy to \eqref{eq:associatedset}) 
is a definable subset of $E$.

A family $(\vh_a)_{a \in A}$ of definable maps $\vh_a : E_a \to \R^m$ 
is called a \emph{definable family} if the map $\vh : E \to \R^m$, where $E$ is the associated set \eqref{eq:associatedset} 
and 
\begin{equation} \label{eq:associatedfunction}
    \vh(a,u) := \vh_a(u), \quad u \in E_a, 
\end{equation}
is definable. 
This is consistent with the first paragraph, since
\begin{align*}
    \Ga(\vh) &= \{(a,u,\vh(a,u)) \in \R^N \times \R^n \times \R^m : (a,u) \in E \} 
    \\
             &= \bigcup_{a \in A} \{(a,u,\vh_a(u)) \in \R^N \times \R^n \times \R^m : u \in E_a \} 
             \\
             &= \bigcup_{a \in A} \{a\} \times  \{(u,\vh_a(u)) \in \R^n  \times \R^m : u \in E_a \}
             = \bigcup_{a \in A} \{a\} \times  \Ga(\vh_a). 
\end{align*}

\subsection{Gromov's inequality}

We need two versions of an inequality due to Gromov \cite{Gromov87}.
We start with a $C^m$-version.

\begin{lemma}[{\cite[Lemma 2]{Kurdyka:1997ab}}] \label{lem:Gromov2KP}
    Let $m\ge 1$. Let $f : I \to \R$ be a $C^{m+1}$-function, where $I = [t_0-r,t_0+r] \subseteq \R$, $r>0$, is an interval. 
    Suppose that, for all $j = 2,\ldots, m+1$, we have either $f^{(j)}\ge 0$ on $I$ or $f^{(j)}\le 0$ on $I$.   
    Then 
    \begin{equation*}
        |f^{(m)}(t_0)| \le 2^{\binom{m+2}{2}-2} \, \frac{\sup_{t \in I} |f(t)|}{r^m}.
    \end{equation*}
\end{lemma}

We combine \Cref{lem:Gromov2KP} with the following lemma in order to get a $C^{m,\om}$-version in \Cref{lem:Gromov2}.

\begin{lemma} \label{lem:Gromov1}
    Let $f : I \to \R$ be a $C^2$-function, where $I = [t_0-r,t_0+r] \subseteq \R$, $r>0$, is an interval, 
    and let $\om$ be a modulus of continuity for $f$. 
    Suppose that $f''\ge 0$ on $I$ or $f''\le 0$ on $I$.   
    Then 
    \begin{equation*}
        |f'(t_0)| \le \frac{\om(r)}{r}.
    \end{equation*}
\end{lemma}

\begin{proof}
    We may assume that $t_0=0$.
    Suppose that $f''\ge 0$ on $I$. 
    Then $f$ is convex and, for $0< s < r$,
    \[
        \frac{f(s) - f(0)}{s} \le \frac{f(r)-f(0)}{r} \le \frac{\om(r)}{r}. 
    \]
    Letting $s \to 0$, we find that $f'(0) \le \om(r)/r$.
    The same reasoning applied to $f(-t)$ shows that also $-f'(0) \le \om(r)/r$ so that the assertion is proved.

    The case $f'' \le 0$ follows from the previous one by considering $-f$.
\end{proof}

\begin{lemma} \label{lem:Gromov2}
    Let $m\ge 1$. Let $f : I \to \R$ be a $C^{m+1}$-function, where $I = [t_0-r,t_0+r] \subseteq \R$, $r>0$, is an interval, 
    and let $\om$ be a modulus of continuity for $f$. 
    Suppose that, for all $j = 2,\ldots, m+1$, we have either $f^{(j)}\ge 0$ on $I$ or $f^{(j)}\le 0$ on $I$.   
    Then 
    \begin{equation*}
        |f^{(m)}(t_0)| \le 2^{\binom{m+1}{2}+m-2} \, \frac{\om(r)}{r^m}.
    \end{equation*}
\end{lemma}

\begin{proof}
    If $m=1$, then the statement is immediate from \Cref{lem:Gromov1}.
    If $m\ge 2$, then,  
    by \Cref{lem:Gromov2KP} applied to $f'$ and in turn \Cref{lem:Gromov1}, we 
    have
    \[
        |f^{(m)}(t_0)| \le \frac{2^{\binom{m+1}{2}-2}}{(\frac{r}2)^{m-1}} \sup_{|t-t_0| \le \frac{r}2 } |f'(t)|
        \le \frac{2^{\binom{m+1}{2}-2}}{(\frac{r}2)^{m-1}}\cdot \frac{\om(\tfrac{r}{2})}{\tfrac{r}{2}} 
        \le 2^{\binom{m+1}{2}+m-2} \, \frac{\om(r)}{r^m}
    \]
    as claimed, since $\om$ is increasing.
\end{proof}

\subsection{Uniform bounds for definable families of functions}

\begin{proposition} \label{prop:try}
    Let $(U_a)_{a \in A}$ be a definable family of open sets $U_a \subseteq \R^k$ 
    and let $U \subseteq \R^N \times \R^k$ be the associated definable set (see \eqref{eq:associatedset}).
    Let $(\vh_a)_{a \in A}$ be a definable family of functions $\vh_a : U_a \to \R$ 
    and let $\vh : U \to \R$ be the associated definable function (see \eqref{eq:associatedfunction}).
    Let $\al \in \N^k$ with $|\al| = p$.
    There exists a definable subset $Z \subseteq U$ such that, for all $a \in A$, 
    $Z_a$ is closed in $U_a$, $\dim Z_a < k$,  
    $\vh_a$ is $C^p$ on $U_a \setminus Z_a$, 
    and, for each open ball $B=B(u,r)$, $r>0$, contained in $U_a \setminus Z_a$, we have
    \begin{equation} \label{eq:try}
        |\p^{\al} \vh_a(u)| \le C(k,p)\, \sup_{v \in B}|\vh_a(v)| \, r^{-|\al|}.
    \end{equation}
\end{proposition}

\begin{proof}
    Consider the definable set 
    \begin{align*}
        X &:= \{(b,v) \in U : (a,u) \mapsto \vh(a,u) \text{ is not } C^{p+1} \text{ at } (b,v) \text{ in } u \}
        \\
          &= \{(b,v) \in U : \vh_b \text{ is not } C^{p+1} \text{ at } v \}
    \end{align*}
    and note that 
    \[
        X_a = \{u \in U_a : \vh_a \text{ is not } C^{p+1} \text{ at } u\}
    \]
    is closed in $U_a$ and, by o-minimality, $\dim X_a < k$. 

    The operator $\p^\al$ is a linear combination of directional derivatives $d_v^p$ for 
    a finite collection $V$ of suitably chosen unit directions $v$ in $\R^k$. 
    Let $\vh_1,\ldots,\vh_s$ be an enumeration of all functions $d^j_v \vh : U \setminus X \to \R$, 
    $j = 2,\ldots,p+1$, $v \in V$ (where $d_v^j$ acts only in the $u$-variable:
    $d_v^j \vh(a,u) = \p_t^j|_{t=0} \, \vh(a,u + tv)$).
    For $i =1,\ldots,s$, set 
    \begin{align*}
        Y_i &:=  \{(a,u) \in U \setminus X : \exists \ep>0 \, \forall v \in B(u,\ep), \vh_i(a,v) = 0 \},
        \\
        Z_i &:= \{(a,u) \in U \setminus X : \vh_i(a,u)=0\} \setminus  Y_i, 
        \\ 
        Z &:= X \cup \bigcup_{i=1}^s Z_i. 
    \end{align*}
    Then $Z$ is a definable subset of $U$ and, for all $a \in A$, $Z_a$ is closed in $U_a$ and $\dim Z_a < k$.

    Now \eqref{eq:try} follows easily by applying \Cref{lem:Gromov2KP} to $t \mapsto \vh(a,u+tv)$. 
\end{proof}

\begin{corollary} \label{cor:La1top}
    Let $(U_a)_{a \in A}$ be a definable family of open sets $U_a \subseteq \R^k$ 
    and let $(\vh_a)_{a \in A}$ be a definable family of $C^1$-functions $\vh_a : U_a \to \R$.
    Suppose that there is a constant $M> 0$ such that
    \begin{equation*}
        |\p_j \vh_a(u)| \le M,\quad a \in A,\, u \in U_a,\,  j =1,\ldots,k.
    \end{equation*}
    Let $p$ be a positive integer.
    There exists a definable family $(Z_a)_{a \in A}$ of closed definable sets $Z_a \subseteq U_a$ of dimension $\dim Z_a<k$ 
    such that, for all $a \in A$, $\vh_a$ is of class $C^p$ on $U_a \setminus Z_a$ and 
    \begin{equation*}
        |\p^\ga \vh_a(u)| \le C(k,p)\, M\, d(u,Z_a \cup \p U_a)^{1 -|\ga|}, \quad u \in U_a \setminus Z_a,
        \, 1 \le |\ga| \le p.
    \end{equation*}
\end{corollary}

\begin{proof}
    Apply \Cref{prop:try} to $\p_j \vh_a$.
\end{proof}

\begin{remark}
    We may assume that $Z_a$ is not empty so that $d(u,Z_a \cup \p U_a)$ is always finite. 
    We will tacitly make this assumption in all subsequent results of this type.
\end{remark}

\begin{proposition} \label{prop:Gromov3}
    Let $(U_a)_{a \in A}$ be a definable family of open sets $U_a \subseteq \R^k$ 
    and $(\vh_a)_{a \in A}$ a definable family of continuous functions $\vh_a : U_a \to \R$.
    Let $p$ be a positive integer.
    Then there exists a definable family $(Z_a)_{a \in A}$ of closed subsets $Z_a \subseteq U_a$ of dimension $\dim Z_a < k$ such that, 
    for all $a \in A$, $\vh_a$ is $C^p$ on $U_a \setminus Z_a$ and
    \begin{equation} \label{eq:Gromov3}
        |\p^\ga \vh_a(x)| \le C(k,p)\, \frac{\om(d(x,Z_a \cup \p U_a))}{d(x,Z_a \cup \p U_a)^{|\ga|}}, \quad x \in U_a \setminus Z_a,
        \, 1 \le |\ga| \le p,
    \end{equation}
    where $\om$ is a modulus of continuity for $\vh_a$.
\end{proposition}

\begin{proof}
    Follow the proof of  \Cref{prop:try} and use \Cref{lem:Gromov2}.
\end{proof}

\begin{remark}
    We want to emphasize that the construction of $(Z_a)_{a \in A}$ is independent of $\om$. 
\end{remark}

\subsection{$\La_p$-regular mappings}
\label{ssec:regmap}

Let $U \subseteq \R^k$ be an open set. 
Let $p$ be a positive integer.
A $C^p$-mapping $\vh : U  \to \R^n$ 
is said to be \emph{$\La_p$-regular} 
if there exists a constant $C>0$ such that 
\begin{equation} \label{eq:Lapreg}
    |\p^\ga \vh(u)| \le C\, d(u,\p U)^{1-|\ga|}, \quad u \in U,\, 1 \le |\ga| \le p.
\end{equation}

$\La_p$-regular maps behave nicely on quasiconvex sets. 
Let us first recall the definition of quasiconvexity.

\begin{definition}[Quasiconvex sets] \label{def:quasiconvex}
    A set $E \subseteq \R^n$ is called \emph{quasiconvex} 
    if there is a constant $C>0$ such any two points $x,y \in E$ can be joined in $E$ 
    by a rectifiable path of length at most $C \, |x-y|$.
\end{definition}

Let $\vh : U  \to \R^n$ by $\La_p$-regular.
If $E \subseteq U$ is a quasiconvex subset, then
$\vh|_{E}$ is Lipschitz on $E$ and extends continuously to a map $\ol \vh$ on $\ol E$.

\subsection{$\La_p$-cells}

Let $p$ be a positive integer.
We define recursively $\La_p$-cells in $\R^n$:

A definable subset $S \subseteq \R^n$ is an \emph{open $\La_p$-cell in $\R^n$} if,
\begin{itemize}
    \item in the case $n=1$, $S$ is an open interval in $\R$, 
    \item in the case $n>1$, $S$ is of the form
        \[
            S= (\ps_1,\ps_2,T) := \{(x',x_n) : x' \in T,\, \ps_1(x') < x_n < \ps_2(x')\}, 
        \]
        where $T$ is an open $\La_p$-cell in $\R^{n-1}$ and each $\ps_i$, $i=1,2$,
        is either a $\La_p$-regular definable function $T\to \R$ or identically $-\infty$ or $+\infty$, 
        and $\ps_1 < \ps_2$ on $T$. (Here $x' = (x_1,\ldots,x_{n-1})$.) 
\end{itemize}
Note that $S$ is quasiconvex. 
If $\ps_i$ is finite, then it is Lipschitz on $T$ and has a continuous extension $\ol \ps_i$ to $\ol T$.

A definable subset $S$ of $\R^n$ is a \emph{$k$-dimensional $\La_p$-cell in $\R^n$}, where $k =0,\ldots,n-1$, 
if 
\[
    S=\{ (u,w) : u \in T,\, w=\vh(u)\} = \Ga(\vh), 
\]
where $u=(x_1,\ldots,x_k)$, $w=(x_{k+1},\ldots,x_n)$, 
$T$ is an open $\La_p$-cell in $\R^k$, and $\vh: T \to \R^{n-k}$ is a $\La_p$-regular definable map.

\begin{definition}[$\La_p$-cell with constant $C$]
    A \emph{$\La_p$-cell $S$ in $\R^n$} is an open or a $k$-dimensional $\La_p$-cell in $\R^n$. 
    We say that $S$ is a $\La_p$-cell in $\R^n$ \emph{with constant $C$} 
    if all the $\La_p$-regular maps involved in the recursive definition of $S$ 
    satisfy \eqref{eq:Lapreg} with the same constant $C$.
\end{definition}

\subsection{Associated functions} \label{ssec:associated}

We associate with any open $\La_p$-cell $S$ in $\R^n$ a sequence of $2n+1$ definable 
functions $\rh_j : \ol S \to [0,\infty]$, for $j=0,1,2,\ldots,2n$.
We put $\rh_0\equiv 1$ and define $\rh_j$ for $j \ge 1$ as follows:
\begin{description}
    \item[Case $n=1$] If $S=(a_1,a_2)$ we set
        \begin{align*}
            \rh_1(x) &:= 
            \begin{cases}
                x-a_1 & \text{ if } a_1 \in \R, 
                \\
                +\infty & \text{ if  } a_1= -\infty,
            \end{cases}
            \\
            \rh_2(x) &:= 
            \begin{cases}
                a_2-x & \text{ if } a_2 \in \R, 
                \\
                +\infty & \text{ if  } a_2= +\infty.
            \end{cases}
        \end{align*}        
    \item[Case $n>1$] If $S = (\ps_1,\ps_2,T)$ and $\si_j$, $j = 1,\ldots,2n-2$, 
        are the functions associated with $T$,
        we set $\rh_j(x)= \si_j(x')$, for $j=1,\ldots,2n-2$, and
        \begin{align*}
            \rh_{2n-1}(x) &:= 
            \begin{cases}
                x_n-\ol \ps_1(x') & \text{ if } \ps_1 \text{ is finite}, 
                \\
                +\infty & \text{ if  } \ps_1\equiv -\infty,
            \end{cases}
            \\
            \rh_{2n}(x) &:= 
            \begin{cases}
                \ol \ps_2(x')-x_n & \text{ if } \ps_2 \text{ is finite}, 
                \\
                +\infty & \text{ if  } \ps_2 \equiv +\infty.
            \end{cases}
        \end{align*}    
\end{description}

\begin{remark}
    We add the function $\rh_0$ (which is not present in \cite{Kurdyka:1997ab,Kurdyka:2014aa,Thamrongthanyalak:2017aa}) 
    in order to handle the extension from unbounded $\La_p$-cells (see the proof of \Cref{thm:mainuniform}).
\end{remark}

Each of the functions $\rh_j$, that is finite, 
is $\La_p$-regular on $S$ and Lipschitz on $\ol S$,  see \cite[Lemma 4]{Kurdyka:1997ab}.
There is a positive constant $C>0$ such that
\begin{equation} \label{eq:K}
    \frac{1}{C} \min_{j\ge 1} \rh_j(x) \le d(x,\p S) \le \min_{j\ge 1} \rh_j(x),\quad x \in \ol S,
\end{equation}
where $d(x,\emptyset) = +\infty$ by convention; see \cite[Lemma 3]{Kurdyka:1997ab}.
Consequently,
\begin{equation} \label{eq:K0}
    \frac{1}{C} \min_{j\ge 0} \rh_j(x) \le \min\{1,d(x,\p S)\} \le \min_{j\ge 0} \rh_j(x),\quad x \in \ol S.
\end{equation}
If $\rh_j$ for $j\ge 1$ is finite, then there exists a positive constant $C>0$ such that 
\begin{equation} \label{eq:rt}
    \Big|\p^\ga \Big( \frac{1}{\rh_j} \Big)(x)\Big| \le C \, d(x,\p S)^{-|\ga|-1}, \quad x \in S,\, |\ga|\le p; 
\end{equation}
see \cite[Lemma 5]{Kurdyka:1997ab} and \eqref{eq:inv}. It follows that for all finite $\rh_j$, $j \ge 0$, we have 
\begin{equation} \label{eq:rt0}
    \Big|\p^\ga \Big(\frac{1}{\rh_j} \Big)(x)\Big| \le C \, \min\{1,d(x,\p S)\}^{-|\ga|-1}, \quad x \in S,\, |\ga|\le p. 
\end{equation}

\begin{remark} \label{rem:Krt}
    The constants $C$ in \eqref{eq:K}--\eqref{eq:rt0} 
    only depend on the constants of the $\La_p$-regular maps involved in the definition of $S$.
\end{remark}

For later reference, we recall that for a non-vanishing $C^p$-function $r$ we have, for $1\le |\ga| \le p$,
\begin{equation} \label{eq:inv}
    \p^\ga (1/r) = \sum_{j=1}^{|\ga|} \Big( \sum_{\substack{\de_1 + \cdots +\de_j=\ga\\ \de_1\ne0,\ldots,\de_j \ne0}} 
    a^\ga_{\de_1, \ldots, \de_j} \p^{\de_1} r \cdots \p^{\de_j} r \Big) r^{-j-1}, 
\end{equation}
where $a^\ga_{\de_1, \ldots, \de_j}$ are integers that only depend on $\ga$ and $\de_1,\ldots,\de_j$. 

\subsection{$\La_p$-stratification of definable sets}

Recall that a definable \emph{$C^p$-stratification} of a definable set $E \subseteq \R^n$ is a finite decomposition $\sS$
of $E$ into definable $C^p$-submanifolds of $\R^n$, called strata,
such that, for each stratum $S \in \sS$,
the frontier $(\p S) \cap E$ in $E$ is the union of some strata of dimension $< \dim S$. 
A stratification is called \emph{compatible} with a collection of finitely many definable subsets of $E$ 
if each subset is a union of strata.

A definable $C^p$-stratification $\sS$ of $E$ is called a \emph{$\La_p$-stratification} if each stratum $S \in \sS$ 
is a $\La_p$-cell in $\R^n$ in some linear coordinate system.

\begin{theorem}[{\cite[Proposition 4]{Kurdyka:1997ab} and \cite[Theorem 3]{Kurdyka:2014aa}}] \label{thm:Laregdecomp}
    Let $E \subseteq \R^n$ be a definable set and let $E_1,\ldots,E_\ell$ be definable subsets of $E$.
    Then there exists a $\La_p$-stratification $\sS$ of $E$ that is compatible with $E_1,\ldots,E_\ell$.
\end{theorem}

\subsection{Uniform $\La_p$-stratifications of definable families of sets} 

We prove a uniform version of \Cref{thm:Laregdecomp}.
Let us first recall a result on uniform $L$-regular decompositions.

\begin{theorem}[{\cite[Proposition 1.4]{KurdykaParusinski06}}] \label{thm:quasiconvex}
    Let $E^i \subseteq \R^N \times \R^n$, where $i \in I$, be a finite collection 
    of definable sets.
    Then there exist finitely many disjoint definable sets $B^j \subseteq \R^N \times \R^n$, where $j \in J$, 
    and linear orthogonal mappings $\vh^j : \R^n \to \R^n$, where $j \in J$, such that:
    \begin{enumerate}
        \item For every $a \in \R^N$, each $\vh^j(B^j_a)$ is a standard $L$-regular cell in $\R^n$ 
            with constant $C=C(n)$.
        \item For every $a \in \R^N$, the family of $B^j_a$, where $j \in J$, is a stratification of $\R^n$.
        \item For any $i \in I$, there exists $J_i \subseteq J$ such that $E^i_a = \bigcup_{j \in J_i} B^j_a$ 
            for every $a \in \R^N$.
    \end{enumerate}
\end{theorem}

Here a \emph{standard $L$-regular cell in $\R^n$ with constant $C=C(n)$} 
(which is terminology used in \cite{KurdykaParusinski06}) is by definition nothing else than a $\La_1$-cell with constant $C=C(n)$.

\begin{definition}[Uniform $\La_p$-stratification]
    Let $(E_a)_{a \in A}$ be a definable family of sets $E_a \subseteq \R^n$ and let $E \subseteq \R^N \times \R^n$ be the associated 
    definable set (see \eqref{eq:associatedset}).
    Let $p$ be a positive integer.

    A finite collection $\sS = \{S^j\}_{j\in J}$ of disjoint definable sets $S^j \subseteq \R^N \times \R^n$ is called    
    a \emph{uniform $\La_p$-stratification of $(E_a)_{a \in A}$} if 
    \begin{enumerate}
        \item there exist linear orthogonal maps $\vh^j : \R^n \to \R^n$, $j \in J$, such that,
            for each $a \in A$ and each $j \in J$, $\vh^j(S^j_a)$ is a $\La_p$-cell in $\R^n$
            with constant $C$ independent of $a \in A$,
        \item for each $a \in A$, the family $S^j_a$, $j \in J$, is a stratification of $E_a$.
    \end{enumerate}
    For all $a \in A$, let $\sS_a := \{S^j_a\}_{j \in J}$.  
    Abusing notation, we will also say that \emph{$(\sS_a)_{a \in A}$ is a uniform $\La_p$-stratification of $(E_a)_{a \in A}$}.

    Let $I$ be a finite index set and, for each $i \in I$, 
    let $(E^i_a)_{a \in A}$ be a definable subfamily of $(E_a)_{a \in A}$. 
    A uniform $\La_p$-stratification $(\sS_a)_{a \in A}$ of $(E_a)_{a \in A}$ is said to be 
    \emph{compatible with $(E^i_a)_{a \in A}$, $i \in I$,} if additionally
    \begin{enumerate}
        \item[(3)] for each $i \in I$, there exists a subset $J_i \subseteq J$ such that $E^i_a = \bigcup_{j \in J_i} S^j_a$ for each $a \in A$.
    \end{enumerate}
\end{definition}

By \Cref{thm:quasiconvex}, there always exist uniform $\La_1$-stratifications. 
We shall see that there exist uniform $\La_p$-stratifications for all $p\ge 1$.

\begin{theorem} \label{thm:uniformLapstrat}
    Let $(E_a)_{a \in A}$ be a definable family of sets $E_a \subseteq \R^n$ and let 
    $(E^i_a)_{a \in A}$, $i\in I$, be a finite collection of definable subfamilies of $(E_a)_{a \in A}$.
    Let $p$ be a positive integer.
    Then there exists a uniform $\La_p$-stratification $(\sS_a)_{a \in A}$ of $(E_a)_{a \in A}$ 
    compatible with $(E^i_a)_{a \in A}$, $i \in I$.
\end{theorem}

\begin{proof}
    Let $k = \max_{a \in A} \dim E_a$. 
    We proceed by induction on $k$.
    If $k=0$, then all $E_a$ are finite and the number of elements of $E_a$ is bounded by a constant independent of $a$. 
    In that case, the assertion is trivially true. 

    Suppose that $k>0$. 
    We claim that 
    there exist a finite collection of disjoint definable sets $T^j \subseteq \R^N \times \R^n$, $j \in J$, 
    and linear orthogonal maps $\vh^j : \R^n \to \R^n$, $j \in J$, 
    such that, for each $a \in A$ and each $j \in J$,
    \begin{itemize}
        \item $T^j_a$ is either empty or open in $E_a$ and compatible with $E^i_a$, $i \in I$,
        \item if $T^j_a \ne \emptyset$ then $\vh^j(T^j_a)$ 
            is a $k$-dimensional $\La_p$-cell in $\R^n$ 
            with constant $C$ independent of $a\in A$, and
        \item $\dim E_a \setminus \bigcup_{j \in J} T^j_a < k$. 
    \end{itemize}
    We allow $T^j_a=\emptyset$ to account for the case $\dim E_a <k$.

    Then we can use the induction hypothesis for the definable family
    $(E'_a)_{a \in A}$, where $E'_a := E_a \setminus \bigcup_{j \in J} T^j_a$, 
    and the definable subfamilies $(E^i_{a} \cap E'_a)_{a \in A}$, $i \in I$, 
    and $((\p T^j_{a}) \cap E'_a)_{a \in A}$, 
    $j \in J$. The statement follows.

    Let us prove the claim.
    \Cref{thm:quasiconvex} implies that the claim holds for $p=1$: 
    let $T^j$, $j \in J$, be the corresponding sets with all the properties as listed in the claim.
    Now we apply \Cref{cor:La1top} and induction on the dimension. 
    In fact, for each fixed $j \in J$, 
    $(T^j_{a})_{a \in A}$ is a definable family of $k$-dimensional $\La_1$-cells $T^j_{a}$
    in $\R^n$
    that are open in $E_a$. 
    After the change of coordinates $\vh^j$, we may assume that $T^j_a$ is a $\La_1$-cell with constant $C$ 
    independent of $a \in A$.
    By \Cref{cor:La1top}, there is a definable family $(Z^j_{a})_{a \in A}$ of closed definable sets $Z^j_{a} \subseteq T^j_{a}$, 
    $\dim Z^j_{a} <k$, such that the functions defining the cell $T^j_{a}$ are $\La_p$-regular with uniform constants independent of $a \in A$
    in the complement of $Z^j_{a}$. 
    Thus there exists a definable family $(S^j_{a})_{a \in A}$ of subsets $S^j_{a} \subseteq T^j_{a}$ 
    such that, for all $a \in A$, $S^j_{a}$ is a finite disjoint union of $k$-dimensional definable $\La_p$-cells $S^{j,\ell}_{a}$ 
    that are open in $E_a$
    with constant $C$ independent of $a \in A$
    and $\dim T^j_{a} \setminus S^j_{a} < k$.
    Thus the number of connected components $S^{j,\ell}_{a}$ of $S^j_{a}$ is uniformly bounded by a constant independent of $a \in A$. 
    This implies the claim.
\end{proof}

\subsection{Consequences}

We may use \Cref{thm:uniformLapstrat} in order to refine \Cref{prop:try}, \Cref{cor:La1top}, and \Cref{prop:Gromov3}.

\begin{corollary} \label{cor:try}
    Let $(U_a)_{a \in A}$ be a definable family of open sets $U_a \subseteq \R^k$.
    Let $(\vh_a)_{a \in A}$ be a definable family of functions $\vh_a : U_a \to \R$.
    Let $p$ be a nonnegative integer.
    There exists a uniform $\La_p$-stratification $(\sS_a)_{a \in A}$ of $(U_a)_{a \in A}$ such that, 
    for all $a \in A$ and each open stratum $S_a \in \sS_a$, $\vh_a$ is $C^p$ on $S_a$ and 
    \[
        |\p^\ga \vh_a(u)| \le C(k,p)\, \frac{\sup\{|\vh_a(v)| : v \in S_a,\, |v-u| < d(u,\p S_a)\}}{d(u,\p S_a)^{|\ga|}}, \quad u \in S_a,
        \, |\ga| \le p.
    \] 
\end{corollary}

\begin{proof}
    This follows from \Cref{thm:uniformLapstrat} and \Cref{prop:try}.
\end{proof}

\begin{corollary} \label{cor:La1top2}
    Let $(U_a)_{a \in A}$ be a definable family of open sets $U_a \subseteq \R^k$.
    Let $(\vh_a)_{a \in A}$ be a definable family of $C^1$-functions $\vh_a : U_a \to \R$.
    Suppose that there is a constant $M> 0$ such that
    \begin{equation*}
        |\p_j \vh_a(u)| \le M,\quad a \in A,\, u \in U_a,\,  j =1,\ldots,k.
    \end{equation*}
    Let $p$ be a positive integer.
    There exists a uniform $\La_p$-stratification $(\sS_a)_{a \in A}$ of $(U_a)_{a \in A}$ such that, 
    for all $a \in A$ and each open stratum $S_a \in \sS_a$, $\vh_a$ is $C^p$ on $S_a$ and 
    \[
        |\p^\ga \vh_a(u)| \le C(k,p)\, M\, d(u,\p S_a)^{1 -|\ga|}, \quad u \in S_a,
        \, 1 \le |\ga| \le p.
    \] 
\end{corollary}

\begin{proof}
    This follows from \Cref{thm:uniformLapstrat} and \Cref{cor:La1top}.
\end{proof}

\begin{proposition} \label{prop:uniformGromov4}
    Let $(U_a)_{a \in A}$ be a definable family of open sets $U_a \subseteq \R^k$.
    Let $(\vh_a)_{a \in A}$ be a definable family of continuous functions $\vh_a : U_a \to \R$. 
    Let $p$ be a positive integer.
    There exists 
    a uniform $\La_p$-stratification $(\sS_a)_{a \in A}$  of $(U_a)_{a \in A}$ such that, 
    for all $a \in A$ and each open stratum $S_a \in \sS_a$, 
    $\vh_a$ is $C^p$ on $S_a$ and
    \[
        |\p^\ga \vh_{a}(u)| \le C(k,p)\, \frac{\om(d(u,\p S_a))}{d(u,\p S_a)^{|\ga|}},\quad  u \in S_a, \, 1\le |\ga|\le p,
    \]
    where $\om$ is a modulus of continuity for $\vh_a$.
\end{proposition}

\begin{proof}
    This follows from \Cref{thm:uniformLapstrat} and \Cref{prop:Gromov3}.
\end{proof}

We will need another uniform fact:

\begin{proposition} \label{prop:forr}
    Let $(U_a)_{a \in A}$ be a definable family of open sets $U_a \subseteq \R^k$. 
    Let $(\vh_a)_{a \in A}$ be a definable family of functions $\vh_a : U_a \to \R$. 
    Let $p$ be a positive integer. 
    There exists a uniform $\La_p$-stratification $(\sS_a)_{a \in A}$ of $(U_a)_{a \in A}$ such that, 
    for all $a \in A$ and each open stratum $S_a \in \sS_a$, 
    $\vh_a$ is $C^p$ on $S_a$ 
    and, for all $j = 1,\ldots, k$,  
    \begin{equation} \label{eq:forr}
        \text{ either }\quad |\p_j \vh_a|\ge 1 \text{ on } S_a \quad \text{ or } \quad |\p_j \vh_a| <1 \text{ on } S_a.
    \end{equation}
\end{proposition}

\begin{proof}
    Let $U \subseteq \R^N \times \R^n$ and $\vh : U \to \R$ be the associated definable set and function (see \eqref{eq:associatedset} 
    and \eqref{eq:associatedfunction}). 
    Let $X \subseteq U$ be the set defined in the proof of \Cref{prop:try}. 

    For $j =1,\ldots,k$, let $\p_j \vh(a,u):=\p_j \vh_a(u)$ and set 
    \begin{align*}
        Y_j &:= \{(a,u) \in U \setminus X : \exists \ep>0 \, \forall v \in B(u,\ep), \p_j\vh(a,v) = 1 \},
        \\
        Z_j &:= \{(a,u) \in U \setminus X : \p_j\vh(a,u)=1\} \setminus  Y_j, 
        \\ 
        Z &:= X \cup \bigcup_{j=1}^k Z_j. 
    \end{align*}
    Then $Z$ is a definable subset of $U$ and, for all $a \in A$, $Z_a$ is closed in $U_a$ and $\dim Z_a < k$.
    Now the statement follows from \Cref{thm:uniformLapstrat}.
\end{proof}

\section{Bounded definable families of Whitney jets} \label{sec:jets}

Recall that a modulus of continuity $\om$ is by definition a positive, continuous, increasing, 
and concave function $\om : (0,\infty) \to (0,\infty)$ 
such that $\om(t) \to 0$ as $t \to 0$.

\subsection{$C^{m,\om}$-functions}

Let $\om$ be a modulus of continuity. Let $U \subseteq \R^n$ be an open set.
Let $C^{0,\om}(U)$ be the set of all continuous bounded functions $f : U \to \R$ such that 
\[
    |f|_{C^{0,\om}(U)} := \inf\{C>0 : |f(x)-f(y)| \le C\, \om(|x-y|) \text{ for all }x,y \in U\} < \infty.
\]
For a nonnegative integer $m$, the set $C^{m,\om}(U)$ consists of all $C^m$-functions such that $\p^\al f$ is 
globally bounded for all $|\al|\le m$ and $\p^\al f \in C^{0,\om}(U)$ for all $|\al|=m$.
Then $C^{m,\om}(U)$ is a Banach space with the norm
\[
    \|f\|_{C^{m,\om}(U)} := \sup_{x \in U} \sup_{|\al|\le m} |\p^\al f(x)| 
    + \sup_{|\al|= m}|\p^\al f|_{C^{0,\om}(U)}.
\]
We say that $f \in C^{m,\om}(U)$ is \emph{$m$-flat} outside an open set $V \subseteq U$ 
if all $\p^\al f$, $|\al| \le m$, 
vanish on $U \setminus V$.

Assume that the open set $U \subseteq \R^n$ is definable. 
We denote by $C^{m,\om}_{\on{def}}(U)$ the subspace of $C^{m,\om}(U)$ 
consisting of the definable functions in the latter space. 
Note that the normed space $C^{m,\om}_{\on{def}}(U)$ is not complete.

\begin{definition}[Bounded families of $C^{m,\om}$-functions]
    Let $m \in \N$ and $\om$ a modulus of continuity.
    A family $(f_a)_{a \in A}$ of $C^{m,\om}$-functions $f_a : U_a \to \R$, where $U_a \subseteq \R^n$ is open, 
    is said to be a \emph{bounded family of $C^{m,\om}$-functions} if 
    \begin{equation*}
        \sup_{a \in A} \|f_a\|_{C^{m,\om}(U_a)} < \infty.
    \end{equation*}
    We say that $(f_a)_{a \in A}$ is a \emph{definable bounded family of $C^{m,\om}$-functions} if
    it is a bounded family of $C^{m,\om}$-functions and, additionally, the families 
    $(U_a)_{a \in A}$ and $(f_a)_{a \in A}$ are definable. 
    Moreover, $(f_a)_{a \in A}$ is called \emph{$m$-flat outside $(V_a)_{a \in A}$} if, 
    for each $a \in A$, $V_a \subseteq U_a$ is open and $f_a$ is $m$-flat on $U_a \setminus V_a$.
    We will say that $(f_a)_{a \in A}$ is \emph{$C^p$ outside $(E_a)_{a \in A}$} if, 
    for each $a \in A$, $E_a \subseteq U_a$ is closed and $f_a$ is $C^p$ on $U_a \setminus E_a$.
\end{definition}

\subsection{Whitney jets of class $C^{m,\om}$}

Let $E$ be a locally closed subset of $\R^n$. 
An \emph{$m$-jet on $E$} is a collection $F=(F^\al)_{|\al|\le m}$ of continuous functions $F^\al : E \to \R$. 
An $m$-jet $F = (F^\al)_{|\al|\le m}$ on $E$ is said to be \emph{flat} on a subset $E' \subseteq E$ 
if all functions $F^\al$, $|\al|\le m$, vanish on $E'$.

For $a \in E$ we denote by 
$T^m_a  F$ the Taylor polynomial
\[
    T^m_a F (x) = \sum_{|\al|\le m} \frac{1}{\al!} F^\al(a) (x-a)^\al, \quad x \in \R^n,
\]
and define the $m$-jet 
\[
    R^m_a F := F - J_E^m (T^m_a F),
\]
where $J_E^m(f) := (\p^\al f|_E)_{|\al|\le m}$ for $f \in C^m(\R^n)$.

A $C^{m,\om}$ (or $C^m$) function $f : \R^n \to \R$ is an \emph{extension to $\R^n$ of $F$} if
$J_E^m(f) = F$. A necessary and sufficient condition for an $m$-jet to have a $C^{m,\om}$-extension to $\R^n$ 
is to be a Whitney jet of class $C^{m,\om}$ (\cite{Whitney34a}, \cite{Glaeser:1958aa}):
by definition,
an $m$-jet $F=(F^\al)_{|\al|\le m}$ on $E$ is a 
\emph{Whitney jet of class $C^{m,\om}$ on $E$} if 
there exists $C>0$ such that 
\begin{equation} \label{eq:W1}
    \sup_{x \in E} \sup_{|\al|\le m} |F^\al(x)| \le C
\end{equation}
and, for all $x,y \in E$ and $|\al|\le m$, 
\begin{equation} \label{eq:W2}
    |(R^m_x F)^\al(y)|\le C \, \om(|x-y|) |x-y|^{m-|\al|}.
\end{equation}

\begin{remark} \label{rem:equivW2}
    A condition equivalent to \eqref{eq:W2} is 
    \begin{equation*} 
        |T^m_x F(z) - T^m_y F(z)| \le C' \, \om(|x-y|) (|z-x|^m+|z-y|^m)
    \end{equation*}
    for all $x,y \in E$ and $z \in \R^n$; see \cite[Proposition IV.1.5]{Tougeron72}.
    Moreover,
    \eqref{eq:W2} holds if and only 
    if
    \begin{align*}
        |F^0(x) - T^m_y F(x)| \le C\, \om(|x-y|) |x-y|^m  \quad \text{ for all $x,y \in E$,}
        \\ \intertext{and, if $m\ge 1$,}
        \p_i F:=(F^{\al+(i)})_{|\al|\le m-1} \text{ is a Whitney jet of class } C^{m-1,\om} \text{ for all } i=1,\ldots,n. 
    \end{align*}
    If $E$ is quasiconvex with constant $C'$ (see \Cref{def:quasiconvex}), then \eqref{eq:W2} follows from 
    \[
        |F^\al(x) - F^\al(y)| \le C''\, \om(|x-y|), \quad x,y \in E,\, |\al| = m;
    \]
    see \cite[IV (2.5.1)]{Tougeron72}; then the constant $C$ in \eqref{eq:W2} depends only on $n$, $m$, $C'$, and $C''$.
\end{remark}

It is not hard to see that the set of all Whitney jets of class $C^{m,\om}$ on $E$ 
with the natural addition and the multiplication $FG := J_E^m(T^m F \cdot T^m G)$
is an $\R$-algebra. 

Let $F$ be an $m$-jet on $E \subseteq \R^n$.
Let $G_1,\ldots,G_n$ be $m$-jets on $A \subseteq \R^k$
such that $(G_1^0,\ldots,G_n^0)(A) \subseteq E$.
The \emph{composite} $F \o G = F\o(G_1,\ldots,G_n)$ of $m$-jets $F$ and $G$ on $A$ is defined by  
\[
    (F\o G)(x) := J_A^m (T^m_{G^0(x)} F \o T^m_x G)(x).
\]
Note that
\[
    T^m_y(F \o G)(x) = \pi_m  \big(T^m_{G^0(y)} F (T^m_y G(x))\big),
\]
where $\pi_m$ is the natural truncation operator (which truncates monomials of order $> m$).
One can show (using \Cref{rem:equivW2}) that, for $m\ge 1$, the composite $F \o G$ is a Whitney jet of class $C^{m,\om}$ 
if $F$ and $G$ are Whitney jets of class $C^{m,\om}$. We will not use this fact, 
but the pullback of Whitney jets of class $C^{m,\om}$ along a $\La_m$-regular map will be crucial; see \Cref{prop:uniformpullback}.

\begin{definition}[Bounded families of Whitney jets of class $C^{m,\om}$]
    A family $(F_a)_{a \in A}$ of Whitney jets $F_a$ of class $C^{m,\om}$ on $E_a \subseteq \R^n$ 
    is said to be a \emph{bounded family of Whitney jets of class $C^{m,\om}$} if 
    the constant $C>0$ in \eqref{eq:W1} and \eqref{eq:W2} can be chosen independent of $a \in A$, that is,
    \begin{align}
    &\sup_{a \in A} \sup_{x \in E_a} \sup_{|\ga|\le m} |F^\ga_a(x)| < \infty \label{eq:uniformW1}
    \intertext{and}
    &\sup_{a \in A} \sup_{x\ne y \in E_a} \sup_{|\ga|\le m} \frac{|(R^m_x F_a)^\ga (y)|}{\om(|x-y|) |x-y|^{m-|\ga|}} < \infty.
    \label{eq:uniformW2}
    \end{align}
    We say that $(F_a)_{a \in A}$ is a \emph{definable bounded family of Whitney jets of class $C^{m,\om}$} if
    it is a bounded family of Whitney jets of class $C^{m,\om}$ and, additionally, the families 
    $(E_a)_{a \in A}$ and $(F^\ga_a)_{a \in A}$, $|\ga| \le m$, are definable.
    We say that it is \emph{flat} on a subfamily $(E'_a)_{a \in A}$ of $(E_a)_{a \in A}$ if 
    $F_a$ is flat on $E'_a$ for all $a \in A$.

    A (definable bounded) family $(f_a)_{a \in A}$ of $C^{m,\om}$-functions $f_a : \R^n \to \R$ 
    is called a \emph{(definable bounded) family of $C^{m,\om}$-extensions to $\R^n$ of $(F_a)_{a \in A}$} if 
    $f_a$ is a $C^{m,\om}$-extension of $F_a$ to $\R^n$, for each $a \in A$.
\end{definition}

\subsection{Separation}

Let $X,Y,Z$ be subsets of $\R^n$.
Recall that $X$ and $Y$ are said to be \emph{$Z$-separated} if there exists $C>0$ such that 
\begin{equation} \label{eq:separation}
    d(x,Y) \ge C\, d(x,Z), \quad x \in X,
\end{equation}
or equivalently, if there is $C'>0$ such that
\[
    d(x,X) + d(x,Y) \ge C' \, d(x,Z),\quad x \in \R^n.
\]
If $X$ and $Y$ are $X \cap Y$-separated, then we will simply say that $X$ and $Y$ are \emph{separated}.

\begin{definition}[Uniformly separated families of sets]
    Let $(X_a)_{a \in A}$, $(Y_a)_{a \in A}$, and $(Z_a)_{a \in A}$ be definable families of subsets of $\R^n$.
    Then $(X_a)_{a \in A}$ and $(Y_a)_{a \in A}$ are said to be \emph{uniformly $(Z_a)_{a \in A}$-separated} 
    if, for all $a \in A$, $X_a$ and $Y_a$ are $Z_a$-separated with a constant $C>0$ (in \eqref{eq:separation}) independent of $a \in A$.
    We will say that $(X_a)_{a \in A}$ and $(Y_a)_{a \in A}$ are \emph{uniformly separated}  
    if they are uniformly $(X_a \cap Y_a)_{a \in A}$-separated.
\end{definition}

\subsection{Pullback along a definable family of $\La_p$-regular maps}

Let $\vh : U \to \R^\ell$ be a $\La_{m+1}$-regular map, where $U \subseteq \R^k$ is open and quasiconvex.
Let $\ol \vh : \ol U \to \R^\ell$ be the continuous extension of $\vh$; see \Cref{ssec:regmap}. 
Consider 
\begin{align*}
    \vh_+ &: U \times \R^\ell \to U \times \R^\ell, \quad (u,w) \mapsto (u,w+ \vh(u)),
    \intertext{and}
    \ol \vh_+ &: \ol U \times \R^\ell \to \ol U \times \R^\ell, \quad (u,w) \mapsto (u,w+ \ol\vh(u)).
\end{align*}
Let $M$ be a closed subset of $U \times \R^\ell$ and $F$ an $m$-jet on $M$. 
The \emph{pullback of $F$ along $\vh_+$} is the $m$-jet
\[
    \vh_+^*F := F \o J^m_N (\vh_+)
\]
on $N:=\vh_+^{-1}(M) = \{(u,w) \in U \times \R^\ell : (u,w + \vh(u)) \in M\}$.

We shall need the following result on the pullback of a definable bounded family of Whitney jets of class $C^{m,\om}$ 
along a definable family $(\vh_a)_{a \in A}$ of $\La_{m+1}$-regular maps. 
For each $a \in A$, let $\vh_{a,+}$ and $\ol \vh_{a,+}$ be defined in analogy to $\vh_+$ and $\ol \vh_+$.

\begin{proposition}[{\cite[Proposition 4.3]{Pawlucki08aa}}] \label{prop:uniformpullback}
    Let $(U_a)_{a \in A}$ be a definable family of open quasiconvex sets $U_a \subseteq \R^k$ with constant (in 
    \Cref{def:quasiconvex}) independent of $a \in A$.
    Let $(\vh_a)_{a \in A}$ be a definable family of $\La_{m+1}$-regular maps $\vh_a : U_a \to \R^\ell$ 
    with constant (in \eqref{eq:Lapreg}) independent of $a \in A$.
    Let $(M_a)_{a \in A}$ be a definable family of  
    closed quasiconvex subsets $M_a$ of $U_a \times \R^\ell$ such that 
    $(\ol M_a)_{a \in A}$ and $(\p U_a \times \R^\ell)_{a \in A}$ 
    are uniformly separated.

    If $(F_a)_{a \in A}$ is a definable bounded family of Whitney jets of class $C^{m,\om}$ on $(\ol M_a)_{a \in A}$
    which is flat on $(\p M_a)_{a \in A}$, 
    then $(G_a)_{a \in A}$ is a definable bounded family of Whitney jets of class $C^{m,\om}$ on $(N_a)_{a \in A}$,
    where
    $G_a:= \vh_{a,+}^* F_a$ and $N_a := \vh_{a,+}^{-1}(M_a)$.
    If, moreover, for each $a \in A$, $t_a : U_a \to (0,\infty)$ is a function satisfying $t_a(u) \le d(u,\p U_a)$ for every 
    $u \in U_a$ and 
    \[
        |F_a^\ka(u,w)|\lesssim  \om(t_a(u)) t_a(u)^{m-|\ka|}, \quad (u,w) \in M_a, \, |\ka| \le m,
    \]
    then, for each $a \in A$, 
    \[
        |G_a^\ka(u,w)|\lesssim \om(t_a(u)) t_a(u)^{m-|\ka|}, \quad (u,w) \in N_a, \, |\ka| \le m.
    \]
\end{proposition}

\begin{proof}
    Follows from the proof of \cite[Proposition 4.3]{Pawlucki08aa}.
\end{proof}

We will be interested in the case that $M_a = \Ga(\vh_a)$, $a \in A$. 
Then $(G_a)_{a \in A}$ extends to a definable bounded family of Whitney jets of class $C^{m,\om}$ on $(\ol N_a = \ol U_a \times 0)_{a \in A}$ 
which is flat on $(\p N_a = \p U_a \times 0)_{a \in A}$.
This follows from the following lemma and Hestenes' lemma (e.g., \cite[Theorem 1.10]{Thamrongthanyalak:2017aa}); see \cite[Remark 4.2]{Pawlucki08aa}.

\begin{lemma} \label{lem:contexofjets}
    Let $(E_a)_{a \in A}$ be a family of locally closed, quasiconvex sets $E_a \subseteq \R^n$ with constant (in 
    \Cref{def:quasiconvex}) independent of $a \in A$.
    Suppose that $(F_a)_{a \in A}$ is a bounded family of Whitney jets of class $C^{m,\om}$ on $(E_a)_{a \in A}$
    such that, for all  $a \in A$ and $|\al|\le m$,
    $F_a^\al$  has a continuous extension $\ol F_a^\al$ to $\ol E_a$.
    Then $(\ol F_a)_{a \in A}$ is a bounded family of Whitney jets of class $C^{m,\om}$ on $(\ol E_a)_{a \in A}$.
\end{lemma}

\begin{proof}
%
    Let $x,y \in \ol E_a$. There exist sequences $(x_k),(y_k) \subseteq E_a$ such that 
    $x_k \to x$ and $y_k \to y$. 
    By assumption, there exist a constant $C_1>0$, independent of $a \in A$ and of $k$,
    and a rectifiable path $\si_k$ joining $x_k$ and $y_k$ in $E_a$ such that 
    for the length of $\si_k$ we have
    \[
        \ell(\si_k) \le C_1\, |x_k - y_k|.
    \]
    Let $F=(F^\al)_{|\al|\le m}$ be a Whitney jet of class $C^{m,\om}$ on $(E_a)_{a \in A}$. 
    By \cite[IV (2.5.1)]{Tougeron72}, for all $|\al| \le m$,
    \[
        |(R^m_{x_k} F_a)^\al (y_k)| 
        \le n^{\frac{m-|\al|}{2}} C_1^{m-|\al|} |x_k-y_k|^{m-|\al|} \sup_{\xi \in \si_k} \sup_{|\be|=m} |F_a^\be(\xi)-F_a^\be(x_k)|.
    \]
    We may assume that $\si_k$ is parameterized by $t \in [0,1]$ with $\si_k(0)=x_k$ and $\si_k(1) = y_k$. 
    By \eqref{eq:uniformW2}, for $t \in [0,1]$,
    \begin{align*}
        \sup_{|\be|=m} |F^\be(\si_k(t))-F^\be(x_k)|
        &\le   C_2 \,  \om(|\si_k(t)-x_k|) \le   C_2 \,  \om(\ell(\si_k|_{[0,t]}))
        \\
        &\le   C_2 \,  \om(\ell(\si_k)) \le   C_2 \,  \om(C_1 |x_k - y_k|) \le   C_3 \,  \om(|x_k - y_k|),
    \end{align*}
    for constants $C_i>0$ independent of $a \in A$.
    Thus
    \[
        |(R^m_{x_k} F_a)^\al (y_k)| 
        \le n^{\frac{m-|\al|}{2}} C_1^{m-|\al|} |x_k-y_k|^{m-|\al|}\, C_3 \,  \om(|x_k - y_k|)
    \] 
    and letting $k \to \infty$ shows that \eqref{eq:uniformW2} is satisfied for $(\ol F_a)_{a \in A}$. 
    It is clear that
    \eqref{eq:uniformW1} is satisfied.
\end{proof}

\subsection{Cutoff}

We finish this section with a technical cutoff result which will be used in the proof of \Cref{thm:mainuniform}.

\begin{proposition}[{\cite[Proposition 3.9]{Pawlucki08aa}}] \label{prop:uniformPawlucki39}
    Let $(U_a)_{a \in A}$ be a definable family of open quasiconvex sets $U_a \subseteq \R^k$ with constant (in 
    \Cref{def:quasiconvex}) independent of $a \in A$.
    Let $(h_a)_{a \in A}$ be a definable family of $C^m$-functions $h_a : U_a \times \R^\ell \to \R$ 
    and $(\rh_a)_{a \in A}$ a definable family of $C^{m+1}$-functions $\rh_a : U_a \to \R$. 
    Let $(t_a)_{a \in A}$ be a definable family of positive Lipschitz functions 
    $t_a : U_a \to (0,\infty)$ with Lipschitz constants independent of $a \in A$ 
    such that $t_a(u) \le d(u,\p U_a)$ for all $a \in A$ and $u \in U_a$.
    For $\ep>0$, consider the definable family $(\De^\ep_{a})_{a \in A}$, where 
    \[
        \De^\ep_{a} := \{(u,w) \in U_a \times \R^\ell : |w| < \ep \, t_a(u)\}.
    \]
    Assume that, for all $a \in A$, 
    \begin{equation} \label{eq:uniformr}
        \Big|\p^\al \Big( \frac{1}{\rh_a}\Big)(u)\Big| \lesssim t_a(u)^{-|\al|-1}, \quad  u \in U_a,\, |\al|\le m+1.  
    \end{equation}
    Let $\xi : \R \to \R$ be a definable $C^m$-function with compact support, fix $1 \le i \le \ell$, and set, 
    for all $a \in A$, 
    \[
        f_a(u,w) := \xi\Big(\frac{w_i}{\rh_a(u)}\Big) h_a(u,w), \quad (u,w) \in U_a \times \R^\ell.
    \]
    If $(h_{a})_{a \in A}$ is a definable bounded family of $C^{m,\om}$-functions on $(\De^\ep_{a})_{a \in A}$ 
    such that, for each $a \in A$,
    \[
        |\p^\ga h_a(u,w)| \lesssim  \om(t_a(u)) t_a(u)^{m-|\ga|}, \quad (u,w) \in \De^\ep_{a},\, |\ga|\le m,
    \]
    then $(f_a)_{a \in A}$ is a definable bounded family of $C^{m,\om}$-functions on $(\De^\ep_{a})_{a \in A}$ 
    such that, for each $a \in A$,

    \[
        |\p^\ga f_a(u,w)| \lesssim  \om(t_a(u)) t_a(u)^{m-|\ga|}, \quad (u,w) \in \De^\ep_{a},\, |\ga|\le m.
    \]
\end{proposition}

\begin{proof}
    It suffices to repeat the proof of Proposition 3.9 in \cite{Pawlucki08aa} (as well as Lemma 3.5 and Proposition 3.6 
    which are used in the proof).
\end{proof}

\begin{remark} \label{rem:uniformPawlucki39}
    \Cref{prop:uniformpullback} and \Cref{prop:uniformPawlucki39} remain true if we remove everywhere the attribute ``definable''. 
\end{remark}

\section{Bounded definable extension of Whitney jets} \label{sec:proof}

This section is devoted to the proof of \Cref{thm:mainuniform}.
Let us recall the setup:

\smallskip\noindent
\emph{Let $0 \le m\le p$ be integers and $\om$ a modulus of continuity.
    Let $(E_a)_{a \in A}$ be a definable family of closed subsets of $\R^n$.
    Let $(F_a)_{a \in A}$ be a definable bounded family of Whitney jets of class $C^{m,\om}$ on $(E_a)_{a \in A}$.
    We will show that
    there exists a definable bounded family $(f_a)_{a \in A}$ of $C^{m,\om}$-extensions to $\R^n$ of $(F_a)_{a \in A}$
that is $C^p$ outside $(E_a)_{a \in A}$.}
\smallskip

For each $a \in A$, let $\supp F_a$ denote the closure of $\bigcup_{|\ka| \le m} \{x \in E_a : F^\ka_a(x) \ne 0\}$ 
and let $(E'_a)_{a \in A}$ be a definable subfamily of $(E_a)_{a \in A}$ consisting of closed subsets $E'_a$ of $E_a$ 
such that $\supp F_a \subseteq E'_a$.

Let $A' := \{a \in A : \supp F_a = \emptyset\}$. 
The family $(F_a)_{a \in A'}$ can be extended by $(0)_{a \in A'}$ to $\R^n$. 
So we may assume that, for all $a \in A$, $\supp F_a \ne \emptyset$ and thus $E'_a \ne \emptyset$.

We proceed by induction on $k := \max_{a \in A} \dim E'_a$ and show:

\begin{hypo}{($\mathbf{I}_k$)}
    Let $(E_a)_{a \in A}$ be a definable family of closed subsets $E_a$ of $\R^n$ 
    and $(F_a)_{a \in A}$ a definable bounded family of Whitney jets of class $C^{m,\om}$ on $(E_a)_{a \in A}$.
    Let $(E'_a)_{a \in A}$ be a definable subfamily of $(E_a)_{a \in A}$ of closed subsets $E'_a$ of $E_a$ such that 
    $\supp F_a \subseteq E'_a$ and $\dim E'_a \le k$, for all $a \in A$.
    Then there exists a definable bounded family $(f_a)_{a \in A}$ of $C^{m,\om}$-extensions to $\R^n$ of $(F_a)_{a \in A}$
    that is $C^p$ outside $(E'_a)_{a \in A}$.
\end{hypo}

Let us fix an integer $p\ge m+1$.
(We need that $p\ge m+1$ in the proof. The case $p=m$ in \Cref{thm:mainuniform} is evidently a trivial consequence.)

\subsection*{Overview of the proof}
Before we actually start the proof of ($\mathbf{I}_k$), 
let us give a brief general overview.
Besides the induction on the dimension $k$, we use, for fixed $k$, an induction on the 
number of the $k$-dimensional strata of $E'_a$. 
This is possible since this number is uniformly bounded independently of $a \in A$, 
thanks to \Cref{thm:uniformLapstrat}.
In this way, we can reduce the proof to the case that $E_a'$ has dimension $k$ 
and is the closure of a single stratum $S_a$ that is the graph of a $\La_p$-regular map $\vh_a$.
We can assume that the Whitney jet $F_a$ is flat on $\p S_a$; see \Cref{prop:afterred}.
This case is then checked in three gradually more general steps:
\begin{enumerate}
    \item In the first step, we assume that $\vh_a \equiv 0$ and $E_a'=E_a$. 
    \item In the second step, we still suppose that $\vh_a \equiv 0$ but allow that $E_a'$ is a proper subset of $E_a$. 
    \item The general case in the final third step is reduced to the previous steps by means of \Cref{prop:uniformpullback}.
\end{enumerate}

\subsection*{Induction basis ($\mathbf{I}_0$)}

If $k=0$, then each $E'_a$ is a finite set (but $E_a$ might be infinite) 
and there is a constant which bounds the number $|E'_a|$ of elements of 
$E'_a$ independently of $a \in A$, by uniform finiteness; see \cite[4.4]{vandenDriesMiller96}. 

Let us make induction on $s := \max_{a \in A} |E'_a|$.
The base case $s =1$ is treated in the following lemma.

\begin{lemma}
    Let $(E_a)_{a \in A}$ be a definable family of closed subsets $E_a$ of $\R^n$ and $(E'_a)_{a \in A}$
    a definable subfamily of $(E_a)_{a \in A}$ such that, for each $a \in A$, $E'_a = \{x_a\}$.
    Let $(F_a)_{a \in A}$ be a definable bounded family of Whitney jets of class $C^{m,\om}$ on $(E_a)_{a \in A}$
    such that $\supp F_a \subseteq \{x_a\}$, for all $a \in A$.
    Then there exists a definable bounded family $(f_a)_{a \in A}$ of $C^{m,\om}$-extensions to $\R^n$ of $(F_a)_{a \in A}$ 
    that is $C^p$ outside $(E'_a)_{a \in A}$.
\end{lemma}

\begin{proof}
    Note that, for each $a \in A$, $x_a$ is an isolated point of $E_a$, by continuity of $F_a$.

    Let $\ch : \R^n \to \R$ be a definable $C^p$-function
    that equals $1$ in a neighborhood of $0$ and
    has support contained in the unit ball.
    For each $a \in A$, set
    \[
        d_a := 
        \begin{cases}
            \min\{1,d(x_a, E_a \setminus \{x_a\})\} & \text{if } E_a \setminus \{x_a\} \ne \emptyset,
            \\
            1 & \text{otherwise.}
        \end{cases}
    \]
    Define, for each $a \in A$, 
    \[
        f_a(x) := \ch\Big(\frac{x-x_a}{d_a}\Big) \cdot T^m_{x_a} F_a(x), \quad x \in \R^n. 
    \]
    Then $(f_a)_{a \in A}$ is a definable family of
    $C^p$-functions $f_a : \R^n \to \R$ such that each $f_a$ has support contained in the ball $B_a:= B(x_a,d_a)$ with radius $d_a$ around $x_a$
    and extends the jet $F_a$.
    We will prove that the family $(f_a)_{a \in A}$ is bounded in $C^{m,\om}(\R^n)$.

    Let $\ga \in \N^n$, where $|\ga|\le m$. Then
    \begin{align*}
        \p^\ga f_a(x) = \sum_{\al +\be = \ga} \binom{\ga}{\al} d_a^{-|\al|}\, \p^\al \ch\Big(\frac{x-x_a}{d_a}\Big) \p^\be T^m_{x_a} F_a(x).
    \end{align*}
    By \eqref{eq:uniformW2} (for $y \in E_a \setminus \{x_a\}$ with $d_a = |x_a -y|$ if $d_a < 1$) and \eqref{eq:uniformW1} (if $d_a = 1$),
    \begin{equation*}
        |F_a^\be(x_a)| \le C \, \om(d_a) d_a^{m - |\be|}, \quad |\be| \le m,
    \end{equation*}
    for a constant $C>0$ independent of $a \in A$.
    For the rest of the proof, $C$ will denote a constant independent of $a \in A$; its actual value may change.
    Thus, for $x \in B_a$,
    \begin{align*}
        |\p^\be T^m_{x_a} F_a(x)| &= \Big|\sum_{|\ka| \le m-|\be|} \frac{1}{\ka!} F^{\ka+\be}_a(x_a) (x-x_a)^{\ka}\Big|
                                  \le C \, \om(d_a) d_a^{m -  |\be|}, \quad |\be| \le m.
    \end{align*}
    It follows that, for all $x \in \R^n$,
    \begin{align} \label{eq:omub}
        |\p^\ga f_a(x)| \le C \, \om(d_a) d_a^{m -  |\ga|} \le C\,\om(1), \quad |\ga| \le m.
    \end{align}

    Now assume that $|\ga|=m$. To see that $|\p^\ga f_a|_{C^{0,\om}(\R^n)}$ is bounded by a constant independent of $a \in A$, 
    it suffices to estimate, for $\al + \be =\ga$, 
    \begin{align*}
        D(x,y):= \Big|d_a^{-|\al|}\, \p^\al \ch\Big(\frac{x-x_a}{d_a}\Big) \p^\be T^m_{x_a} F_a(x) -
        d_a^{-|\al|}\, \p^\al \ch\Big(\frac{y-x_a}{d_a}\Big) \p^\be T^m_{x_a} F_a(y)\Big|.
    \end{align*}

    Let us first assume that $x,y \in \ol B_a$. Then 
    \begin{align*}
        \MoveEqLeft    d_a^{-|\al|}\, \Big| \p^\al \ch\Big(\frac{x-x_a}{d_a}\Big) -
        \p^\al \ch\Big(\frac{y-x_a}{d_a}\Big) \Big| |\p^\be T^m_{x_a} F_a(x)|
        \\
    &\le C\, \frac{\om(d_a)}{d_a} |x-y| \le 2C\, \om(|x-y|),
    \end{align*}
    since $\om$ is concave and $|x-y| \le 2 d_a$. 
    On the other hand,
    \begin{align*}
        \MoveEqLeft    d_a^{-|\al|}\, \Big|\p^\al \ch\Big(\frac{y-x_a}{d_a}\Big) \Big| |\p^\be T^m_{x_a} F_a(x)- \p^\be T^m_{x_a} F_a(y)|
        \\
        &\le C\, d_a^{-|\al|}\, \sum_{|\ka| \le m-|\be|} \frac{1}{\ka!} |F^{\ka+\be}_a(x_a)| |(x-x_a)^{\ka} - (y-x_a)^\ka|
        \\
        &\le C'\, \frac{\om(d_a)}{d_a}\, |x-y| \le 2C'\, \om(|x-y|).  
    \end{align*}
    So $D(x,y) \le C\, \om(|x-y|)$ for a constant $C>0$ independent of $a \in A$.

    If $x$ and $y$ lie outside of $B_a$, then $D(x,y)=0$.
    If $x \in B_a$ and $y \not\in \ol B_a$ and $z$ is the point, where the line segment $[x,y]$ meets $\p B_a$, then 
    \[
        D(x,y) = D(x,z) \le C\, \om(|x-z|) \le C\, \om(|x-y|).
    \]
    This ends the proof.
\end{proof}

Assume that $s >1$. 
For each $a \in A$, choose a numbering of the elements of $E'_a = \{x_{a,1},\ldots,x_{a,s_a}\}$, where $s_a \le s$.
By the induction hypothesis, 
$(F_a|_{E_a \setminus \{x_{a,2},\ldots,x_{a,s_a}\}})_{a\in A}$ 
admits a definable bounded family $(f^1_a)_{a \in A}$ of $C^{m,\om}$-extensions to $\R^n$ 
that is $C^p$ outside $(\{x_{a,1}\})_{a \in A}$.
Then $(F_a - J^m_{E_a}(f^1_a))_{a \in A}$ is a definable bounded family of 
Whitney jets of class $C^{m,\om}$ on $(E_a)_{a \in A}$ which is flat on $(E_a \setminus \{x_{a,2},\ldots,x_{a,s_a}\})_{a \in A}$
and has a definable bounded family $(f^2_a)_{a \in A}$ of $C^{m,\om}$-extensions to $\R^n$
that is $C^p$ outside $(\{x_{a,2},\ldots,x_{a,s_a}\})_{a \in A}$, again by the induction hypothesis.
Thus, $(f^1_a + f^2_a)_{a \in A}$ is the desired definable bounded family of $C^{m,\om}$-extensions to $\R^n$ of 
$(F_a)_{a \in A}$.

This ends the induction on $s$ and the base case ($\mathbf{I}_0$) of the induction on $k$.

\subsection*{Setup for the induction step}

Let $k>0$ and suppose that ($\mathbf{I}_{k-1}$) holds. We will prove ($\mathbf{I}_k$). 
This will be accomplished by showing \Cref{prop:afterred} below, but first we make 
a few preparatory reductions. 


By \Cref{thm:uniformLapstrat}, 
there is a uniform $\La_p$-stratification $(\sS_a)_{a \in A}$ of $(E_a)_{a \in A}$ 
compatible with $(E'_a)_{a \in A}$
such that, for each $a \in A$ and each $|\ka| \le m$, $F^\ka_a$ is of class $C^p$ on the strata in $\sS_a$.

By ($\mathbf{I}_{k-1}$), we may assume that $\dim E'_a = k$ for all $a \in A$ and
there is a definable bounded $C^{m,\om}$-extension $(f^0_a)_{a \in A}$ to $\R^n$ 
of the restriction of $(F_a)_{a \in A}$ 
to $(E_a \setminus P_a)_{a \in A}$, where 
\[
    P_a = \bigcup \{S_a \in \sS_a : S_a \subseteq E'_a, \, \dim S_a = k\}.
\]
Replacing $F_a$ by $F_a - J^m_{E_a}(f^0_a)$, for each $a \in A$, we may assume that 
$F_a$ is flat on all strata $S_a \in \sS_a$, $S_a \subseteq E'_a$, with $\dim S_a< k$ 
and also on $E_a \setminus E'_a$.

Let us now see that we may furthermore reduce to the case that, for each $a\in A$, 
$E'_a$ is the closure of just one $k$-dimensional stratum $S_a$
and that $F_a$ is flat on its frontier.
Indeed, the number $s_a$ of $k$-dimensional strata of $E'_a$ is uniformly bounded by a constant not depending on $a \in A$.
We may use induction on $s := \max_{a \in A} s_a$ of which the above statement is the base case that 
we take for granted for the moment. 
The induction step works just as for finite sets $E'_a$:
for each $a \in A$, let $S_{a,1},\ldots,S_{a,s_a}$ be a numbering of the $k$-dimensional strata of $E'_a$.
By the induction hypothesis, 
$(F_a|_{E_a \setminus R_a})_{a \in A}$, where $R_a := \bigcup_{i\ge 2} S_{a,i}$, 
admits a definable bounded family $(f^1_a)_{a \in A}$ 
of $C^{m,\om}$-extensions to $\R^n$ that is $C^p$ outside $(E'_a \setminus R_a)_{a \in A}$. 
Then $(F_a - J^m_{E_a}(f^1_a))_{a \in A}$ is a definable bounded family of 
Whitney jets of class $C^{m,\om}$  on $(E_a)_{a \in A}$ 
which is flat on $(E_a \setminus R_a)_{a \in A}$  
and has a definable bounded family $(f^2_a)_{a \in A}$ of $C^{m,\om}$-extensions to $\R^n$ 
that is $C^p$ outside $(R_a)_{a \in A}$, 
again by the induction hypothesis.
Thus, $(f^1_a + f^2_a)_{a \in A}$ is the desired definable bounded family of $C^{m,\om}$-extensions to $\R^n$ of 
$(F_a)_{a \in A}$.

In the case that $k = n$, $S_a$ is open in $\R^n$ and
extending $F_a$ by $0$ outside $S_a$, for all $a \in A$, yields a definable bounded family $(F_a)_{a \in A}$ of Whitney jets of class $C^{m,\om}$ 
on $(\R^n)_{a \in A}$ 
so that 
$(F^0_a)_{a \in A}$ is the desired family of $C^{m,\om}$-extensions. 
This follows
from Hestenes' lemma (e.g., \cite[Theorem 1.10]{Thamrongthanyalak:2017aa});
indeed, if $x \in S_a$ and $y \not\in \ol S_a$ and $z$ is the point, where the line segment $[x,y]$ meets $\p S_a$, then, 
by \eqref{eq:uniformW2} and \Cref{rem:equivW2}, 
for any $u \in \R^n$,
\begin{align*}
    |T^m_x F_a(u) - T^m_y F_a(u)| &= |T^m_x F_a(u)|= |T^m_x F_a(u) - T^m_z F_a(u)|
    \\ 
                                  &\le C\, \om(|x-z|)(|u-x|^m + |u-z|^m)
                                  \\ 
                                  &\le 2C\, \om(|x-y|)(|u-x|^m + |u-y|^m),
\end{align*}
for a constant $C>0$ independent of $a \in A$,
since $|u-z| \le \max\{|u-x|,|u-y|\}$.

Consequently, we may assume that $\ell := n -k >0$.

We reduced the proof to showing the following.
(We may assume that $S_a$ is a $\La_p$-cell in a fixed orthogonal system of coordinates of $\R^n$, 
which is independent of $a \in A$,
thanks to \Cref{thm:uniformLapstrat}.)

\begin{proposition} \label{prop:afterred}
    Let $(E_a)_{a \in A}$ be a definable family of closed sets $E_a$ in $\R^n$.
    Let $(E'_a)_{a \in A}$ be a definable subfamily of $(E_a)_{a \in A}$
    of
    closed subsets $E'_a$ of $E_a$ with $\dim E'_a =k$ 
    such that $E'_a = \ol S_a$, where 
    \[
        S_a = \{(u,\vh_a(u)) \in \R^k \times \R^\ell : u \in T_a\} = \Ga(\vh_a),
    \]
    and $(\vh_a)_{a \in A}$ is a definable family of $\La_p$-regular maps 
    $\vh_a : T_a \to \R^\ell$, 
    $T_a$ an open $\La_p$-cell in $\R^k$, 
    and all constants in the definition of $T_a$ and $\vh_a$ are independent of $a \in A$.
    Then any definable bounded family $(F_a)_{a \in A}$ of Whitney jets of class $C^{m,\om}$ on 
    $(E_a)_{a \in A}$ such that, for all $a \in A$, $\supp F_a \subseteq E'_a$, $F_a$ is flat on $\p S_a$, 
    and $F_a^\ka$, $|\ka|\le m$, is $C^p$ on $S_a$, 
    admits a definable bounded family $(f_a)_{a \in A}$ of $C^{m,\om}$-extensions to $\R^n$ 
    that is $C^p$ outside $(E'_a)_{a \in A}$.
\end{proposition}

The proposition is proved in three gradually more general steps:
\begin{description}
    \item[Step 1] $\vh_a \equiv 0$ and $E'_a = E_a$ for all $a \in A$. 
    \item[Step 2] $\vh_a \equiv 0$ for all $a \in A$.
    \item[Step 3] The general case.
\end{description}

\subsection*{Step 1}

For all $a \in A$, $E_a = E'_a = \ol T_a \times 0$, where $T_a \subseteq \R^k$ is an open $\La_p$-cell with constant $C$ 
independent of $a \in A$.
We will prove \Cref{prop:afterred} in this special case 
with the additional property that 
$(f_a)_{a \in A}$ is $m$-flat outside $(\De(T_a \times 0))_{a \in A}$, where
\begin{equation} \label{eq:De}
    \De(T_a \times 0) :=  \{(u,w) \in T_a \times \R^\ell : |w| < \min\{1,d(u, \p T_a)\}\}.
\end{equation}

For each $a \in A$, we write 
\[
    F_a=(F_a^{(\al,\be)})_{(\al,\be) \in \N^k \times \N^{\ell},|\al|+|\be|\le m}.
\]
Fix $\be \in \N^{\ell}$ with $|\be| \le m$.
Let $F_{a,\be}$ be the $m$-jet which results from $F_a$ by setting all $F_a^{(\al,\be')}$ equal to $0$ whenever $\be' \ne \be$. 
Then $(F_{a,\be})_{a \in A}$ is a definable bounded family of Whitney jets of class $C^{m,\om}$ on $(\ol T_a \times 0)_{a \in A}$. 
Indeed, for each $a \in A$, the definable Whitney jet $F_a$ of class $C^{m,\om}$ on $\ol T_a \times 0$
can be identified with a collection $\tilde F_{a,\be}$, $|\be|\le m$,  
where $\tilde F_{a,\be}$ is a definable $C^{m-|\be|,\om}$-function on $\ol T_a$ such that 
$\p^\al \tilde F_{a,\be}(u) = F_a^{(\al,\be)}(u,0)$ for all $u \in \ol T_a$ and $\al \in \N^k$, $|\al| \le m-|\be|$;
see \cite[Remark 5]{Kurdyka:2014aa} and \cite[pp. 87-88]{Glaeser:1958aa}.
It suffices to prove that, for each $\be$, $(F_{a,\be})_{a \in A}$ admits a definable bounded family of $C^{m,\om}$-extensions to $\R^n$ 
that is $m$-flat outside $(\De(T_a \times 0))_{a \in A}$ and
$C^p$ outside $(\ol T_a \times 0)_{a \in A}$. 
Thus, we may suppose that, for each $a \in A$,  $F_a^{(\al,\be')} =0$ whenever $\be' \ne \be$. 
By assumption, $F_a^{(\al,\be)}$ is $C^p$ on $T_a \times 0$.

By \Cref{thm:uniformLapstrat}, \Cref{cor:try}, and \Cref{prop:uniformGromov4}, 
there is a uniform $\La_p$-stratification $(\sD_a)_{a \in A}$ of $(\ol T_a)_{a \in A}$ 
such that, for all $a \in A$, each open $k$-dimensional $D_a \in \sD_a$, and all $\al,\be$, $F^{(\al,\be)}_a$ is $C^p$ on $D_a \times 0$, and, 
for all $u \in D_a$ and $\ga\in \N^k$ with $|\ga| \le p$, 
we have 
\begin{equation} \label{eq:uniformGromovsup}
    |\p^\ga F_a^{(\al,\be)}(u,0)| \le L\, \frac{\sup\{|F_a^{(\al,\be)}(v,0)| : v \in D_a,\, |v-u|< d(u,\p D_a)\}}{d(u,\p D_{a})^{|\ga|}},
\end{equation}
and, if $1 \le |\ga| \le p$, 
\begin{equation} \label{eq:uniformGromov5}
    |\p^\ga F_a^{(\al,\be)}(u,0)| \le L\, \frac{\om(d(u,\p D_{a}))}{d(u,\p D_{a})^{|\ga|}},
\end{equation}
where $L>0$ is a constant independent of $a \in A$. 

For each $a \in A$, let $Z_a := \bigcup \{D_a  \in \sD_a : \dim D_a < k\}$.
Setting
\[
    G_a(x) := 
    \begin{cases}
        F_a(x) & \text{if } x \in Z_a \times 0,
        \\
        0 & \text{if } x \in \R^n \setminus \De(T_a \times 0),
    \end{cases}
\]
defines a definable bounded family $(G_a)_{a \in A}$ of Whitney jets of class $C^{m,\om}$
on $((Z_a \times 0) \cup (\R^n \setminus \De(T_a \times 0)))_{a \in A}$. 
This follows from Hestenes' lemma (e.g., \cite[Theorem 1.10]{Thamrongthanyalak:2017aa}) and the following reasoning.
Clearly, $(G_a)_{a \in A}$ satisfies \eqref{eq:uniformW1}.
To see \eqref{eq:uniformW2} it suffices to consider the case that $x \in Z_a \times 0$ and $y \in \R^n \setminus \De(T_a \times 0)$
and to show that 
\begin{equation} \label{eq:zeroext}
    |F_a^\ka(x)| \le C\, \om(|x-y|) |x-y|^{m-|\ka|}, \quad |\ka|\le m,
\end{equation}
for a constant $C>0$ independent of $a \in A$.
We have 
\begin{equation*}
    |F_a^\ka(x)| \le C\, \om(d(u, \p T_a)) d(u, \p T_a)^{m-|\ka|}, \quad |\ka|\le m,
\end{equation*}
by \eqref{eq:uniformW2}, since $(F_a)_{a \in A}$ is flat on $(\p T_a \times 0)_{a  \in A}$,
and
\begin{equation} \label{eq:divom1}
    |F_a^\ka(x)| \le C = \frac{C}{\om(1)} \cdot \om(1) 1^{m -|\ka|}, \quad |\ka|\le m,
\end{equation}
by \eqref{eq:uniformW1}.
Then \eqref{eq:zeroext} follows, since
we have $|x-y| \ge c\,\min\{1,d(u, \p T_a)\}$ for a universal constant $c>0$, by \eqref{eq:De}.

By ($\mathbf{I}_{k-1}$), 
there exists a definable bounded family $(g_a)_{a \in A}$ of $C^{m,\om}$-extensions 
to $\R^n$ of $(G_a)_{a \in A}$ that is $C^p$ outside 
$(Z_a \times 0)_{a \in A}$.  
So, instead of $(F_a)_{a \in A}$, 
it is enough to consider $(F_a - J^m_{E_a}(g_a))_{a \in A}$. 

If $D_{a}$ and $D'_{a}$ are distinct open $k$-dimensional strata in $\sD_a$, 
then $\De(D_a \times 0) \subseteq \De(T_a \times 0)$ and $\ol{\De(D_a \times 0)} \cap \ol{\De(D'_a \times 0)} \subseteq Z_a \times 0$.
Thus it suffices to find, separately for each $(D_a)_{a \in A}$, 
a definable bounded family $(f_a)_{a \in A}$ of $C^{m,\om}$-extensions to $\R^n$ 
of $((F_a -J^m_{E_a}(g_a))|_{\ol D_a \times 0})_{a \in A}$ 
that is $m$-flat 
outside $(\De(D_a \times 0))_{a \in A}$ and 
$C^p$ outside $(\ol D_a \times 0)_{a \in A}$.

For each $a \in A$, set 
\begin{equation*}
    h_a(u,w) := \frac{1}{\be!} F_a^{(0,\be)}(u,0) w^\be - g_{a}(u,w),
\end{equation*}
and define $f_a : \R^n \to \R$ by
\begin{equation*}
    f_a(u,w) := 
    \begin{cases}
        r_a(u,w) h_a(u,w) & \text{ if } u \in D_{a},
        \\
        0 & \text{ otherwise, } 
    \end{cases}
\end{equation*}
where 
\begin{equation} \label{eq:cutoff} 
    r_a(u,w) := \prod_{i=1}^{\ell} \prod_{j=0}^{2k} \xi\Big(C \, \sqrt{\ell} \frac{w_i}{\rh_{a,j}(u)}\Big)
\end{equation}
with $\xi : \R \to \R$ a semialgebraic $C^p$-function which is $1$ near $0$ and vanishes outside 
$(-1,1)$,
$\rh_{a,0},\rh_{a,1},\ldots,\rh_{a,2k}$ the functions associated with the open $\La_p$-cell $D_{a}$ (see \Cref{ssec:associated}), 
and $C$ is the constant 
from \eqref{eq:K0} which may be taken independent from $a \in A$, 
since it is determined by the constants in the definition of the $\La_p$-cells $D_a$, $a \in A$; see \Cref{rem:Krt}.
Note that the $m$-jet of $h_a$ at $(u,0)$ coincides with $(F_a -J^m_{E_a}(g_a))(u,0)$ 
for all $u \in D_a$.

By construction, $(f_a)_{a \in A}$ is a definable family.
We will see that it is a bounded family of $C^{m,\om}$-extensions to $\R^n$ of $((F_a -J^m_{E_a}(g_a))|_{\ol D_a \times 0})_{a \in A}$. 
It is
$m$-flat outside $(\De(D_{a} \times 0))_{a \in A}$, thanks to the properties of $r_a$,  
and it is $C^p$ outside $(\ol D_{a} \times 0)_{a \in A}$. 
Indeed, if $(u,w) \in (D_a \times \R^\ell) \setminus \De(D_{a} \times 0)$, then, by \eqref{eq:K0}, 
\begin{align*}
    \sqrt{\ell} \max_{1 \le i \le \ell} |w_i|\ge |w| \ge \min\{1,d(u,\p D_a\} \ge \frac{1}{C} \min_{0 \le j\le 2k} \rh_{a,j}(u)
\end{align*}
so that $r_a$ is identically zero on $(D_a \times \R^\ell) \setminus \De(D_{a} \times 0)$. 
It remains to check that the family $(f_a)_{a \in A}$ is contained and bounded in $C^{m,\om}(\R^n)$.
To this end, we need two lemmas.

\begin{lemma} \label{lem:uniformStep1.2}
    For each $a \in A$, $h_a$ is of class $C^{m,\om}$ on $\De(D_a \times 0)$ and 
    the $C^{m,\om}$-norm of $h_a$ on $\De(D_{a} \times 0)$ is bounded by a constant 
    independent of $a \in A$. 
\end{lemma}

\begin{proof}
    By construction, each $h_a$ is of class $C^m$.
    Since $(g_a)_{a \in A}$ is a bounded family of $C^{m,\om}$-functions on $\R^n$, it suffices 
    to consider $(u,w) \mapsto F_a^{(0,\be)}(u,0) w^\be$.
    We have to 
    check that
    there is a constant $C>0$ such that, 
    for all $a \in A$, all
    $\ka = (\si,\ta) \in \N^k \times \N^\ell$, $|\ka| \le m$, and all $(u,w) \in \De(D_a \times 0)$,
    \begin{equation} \label{eq:Cmnorm}
        |\p^\ka \big( F_a^{(0,\be)}(u,0)w^\be\big)| \le C, 
    \end{equation}
    and, if $|\ka| = m$, for all $x_i = (u_i,w_i) \in \De(D_a \times 0)$, $i=1,2$,
    \begin{equation} \label{eq:Hnorm}
        |\p^\ka \big( F_a^{(0,\be)}(u_1,0)w_1^\be\big) - \p^\ka \big( F_a^{(0,\be)}(u_2,0)w_2^\be\big)|
        \le C\, \om(|x_1 -x_2|). 
    \end{equation}
    Fix $\ka = (\si,\ta)$ with $|\ka|\le m$.
    We may assume that $\ta \le \be$.
    Let us decompose $\si$ as $\si = \al + \ga$, where $\al,\ga \in \N^k$, 
    $|\al| \le m -|\be|$, and $\al$ is maximal with this property. 
    Thus, if $|\ga| >0$ then $|\al| + |\be| = m$. 
    To see \eqref{eq:Cmnorm}, observe that, by \eqref{eq:uniformGromovsup}, \eqref{eq:uniformW1}, and $|w| < \min\{1,d(u,\p D_a)\}$,
    \begin{align*}
        |\p^\ga F_a^{(\al,\be)}(u,0)w^{\be-\ta}| &\le  L\, \frac{\sup\{|F_a^{(\al,\be)}(v,0)| : v \in D_a,\, |v-u|< d(u,\p D_a)\}}{d(u,\p D_{a})^{|\ga|}}
        |w|^{|\be-\ta|}
        \\
                                                 &\le C L\, |w|^{|\be -\ta|-|\ga|} \le CL,
    \end{align*}
    where $C>0$ is the supremum in \eqref{eq:uniformW1};
    indeed, if $\ga \ne 0$ then  $|\al|+|\be| = m$ and 
    thus $|\be - \ta| \ge |\ga|$.

    Let us prove \eqref{eq:Hnorm}.
    Now $|\ka| = m$ and $|\al|+|\be| = m$, whence
    $|\be - \ta| = |\ga|$. 
    Then it is enough to show
    \begin{equation} \label{eq:toshow}
        |\p^\ga F_a^{(\al,\be)}(u_1,0)w_1^{\be-\ta} - \p^\ga  F_a^{(\al,\be)}(u_2,0)w_2^{\be-\ta}|
        \le C\, \om(|x_1 -x_2|). 
    \end{equation}
    If $\ga =0$, this follows from \eqref{eq:uniformW2}.
    So let us assume that $|\ga|\ge 1$.

    Set $t_a(u):= \tfrac{1}{2} d(u,\p D_a)$.
    Then 
    \begin{equation} \label{eq:tLip}
        |t_a(u_1) - t_a(u_2)| \le \tfrac{1}{2}|u_1 - u_2|.
    \end{equation}  
    Note that, for $i=1,2$, 
    \begin{equation} \label{eq:wtu}
        |w_i| <  d(u_i,\p D_a) = 2 t_a(u_i).
    \end{equation}

    We consider two cases.

    \subsubsection*{Case 1} Suppose that $t_a(u_i) \le |x_1 - x_2|$ for $i =1,2$.
    Then, by \eqref{eq:uniformGromov5} and \eqref{eq:wtu},
    \begin{align*}
        | \p^\ga F_a^{(\al,\be)}(u_i,0)w_i^{\be-\ta}| 
        \le L\,\om(2t_a(u_i)) 
        \le 2L\, \om(|x_1-x_2|), 
    \end{align*}
    since $\om$ is concave and increasing.

    \subsubsection*{Case 2} Assume (without loss of generality) that $t_a(u_1)> |x_1-x_2|$. 
    Then $|u_1 - u_2| \le |x_1-x_2| < t_a(u_1) = \tfrac{1}{2} d(u_1,\p D_a)$ 
    so that the line segment $[x_1,x_2]$ is contained in $D_a \times \R^\ell$.
    Furthermore, if $u \in [u_1,u_2]$ then, by \eqref{eq:tLip},
    \[
        |t_a(u_1) - t_a(u)| \le  \tfrac{1}{2}|u_1 - u| \le \tfrac{1}{2} |x_1 - x_2| < 
        \tfrac{1}{2} t_a(u_1),
    \]
    whence
    \[
        \tfrac{1}{2}  t_a(u_1) < t_a(u) < \tfrac{3}{2} t_a(u_1), \quad u \in [u_1,u_2].
    \]
    The left-hand side of \eqref{eq:toshow} is bounded by 
    \begin{align*}
        |\p^\ga F_a^{(\al,\be)}(u_1,0) - \p^\ga  F_a^{(\al,\be)}(u_2,0)| |w_1|^{|\be-\ta|}  
        + |\p^\ga  F_a^{(\al,\be)}(u_2,0)| |w_1^{\be -\ta} - w_2^{\be -\ta}|.
    \end{align*}
    By \eqref{eq:uniformGromov5} and \eqref{eq:wtu},
    \begin{align*}
        \MoveEqLeft
        |\p^\ga F_a^{(\al,\be)}(u_1,0) - \p^\ga  F_a^{(\al,\be)}(u_2,0)||w_1|^{|\be-\ta|} 
        \\
    &\lesssim \sup_{u \in [u_1,u_2]} \sum_{j=1}^k |\p^{\ga+(j)} F_a^{(\al,\be)}(u,0)| |u_1 - u_2| \, t_a(u_1)^{|\ga|}
    \\
    &\lesssim \sup_{u \in [u_1,u_2]}  \frac{\om(2t_a(u))}{t_a(u)^{|\ga|+1}} |u_1 - u_2|\, t_a(u_1)^{|\ga|}
    \\
    &\lesssim  \frac{\om(t_a(u_1))}{t_a(u_1)} |x_1 - x_2|
    \\
    &\le \om(|x_1 - x_2|),
    \end{align*}
    since $\om$ is concave.
    Again, by \eqref{eq:uniformGromov5} and \eqref{eq:wtu},
    \begin{align*}
        \MoveEqLeft 
        |\p^\ga  F_a^{(\al,\be)}(u_2,0)| |w_1^{\be -\ta} - w_2^{\be -\ta}|
        \\
    &\lesssim \frac{\om(2t_a(u_2))}{t_a(u_2)^{|\ga|}} |w_1 -w_2|\, t_a(u_1)^{|\ga|-1}
    \lesssim \frac{\om(t_a(u_1))}{t_a(u_1)} |x_1 -x_2| \le \om(|x_1-x_2|).
    \end{align*}
    The proof is complete.
\end{proof}

The proof shows that \eqref{eq:Hnorm} actually holds on the larger set $\{(u,w) \in D_a \times \R^\ell : |w| < d(u,\p D_a)\}$.

\begin{lemma} \label{lem:uniformStep1.1}
    For each $a \in A$,
    \begin{equation} \label{eq:uniformStep1.1}
        |\p^\ka h_a(u,w)| \le C\, \om(d(u,\p D_{a})) d(u,\p D_{a})^{m-|\ka|}, 
    \end{equation}
    for all $(u,w) \in \De(D_{a} \times 0)$, all 
    $\ka \in \N^n$, $|\ka| \le m$, 
    and a constant $C>0$ independent of $a \in A$. 
\end{lemma}

\begin{proof}
    Fix $x= (u,w) \in \De(D_a \times 0)$.
    If $d(u,\p D_a) < d(u,\p T_a)$, then let $u' \in \p D_a$ be such that $|u-u'| = d(u,\p D_a)$ 
    and set $x' = (u',0)$.
    The open line segment $(x,x')$ 
    is contained in $\De(D_a \times 0)$.
    Since $u' \in T_a$, where $F_a^{(0,\be)}$ is of class $C^p$, and $h_a$ is of class $C^{m,\om}$ on $\De(D_a \times 0)$ 
    with $C^{m,\om}$-norm bounded by a constant independent of $a \in A$, by \Cref{lem:uniformStep1.2},  
    we may conclude the assertion from Taylor's theorem.

    So we assume that $d(u,\p D_a) = d(u,\p T_a)$. 
    Let $u' \in \p T_a$ such that $|u-u'| = d(u,\p T_a)$.
    Let us first assume that  $\ka = (\si,\ta) \in \N^k \times \N^\ell$ with $|\ka| = m$.
    By construction,  $\p^\ka g_a(u',0)=0$ so that
    \[
        |\p^\ka g_a(u,w)| = |\p^\ka g_a(u,w) - \p^\ka g_a(u',0)| \lesssim \om(|u-u'|) 
        = \om(d(u,\p D_a)), 
    \]
    where we used that $|w| < d(u,\p D_a) = |u-u'|$. 
    Hence it suffices to consider 
    $\p^\ka \big(F_a^{(0,\be)}(u,0)w^\be \big)$ or equivalently 
    $\p^\ga F_a^{(\al,\be)}(u,0)w^{\be-\ta}$, where
    $\al,\ga \in \N^k$ are such that $\al+\ga = \si$, $|\al|+|\be|=m$, and $\ta \le \be$.
    Thus $|\be-\ta|=|\ga|$. 
    If $|\ga|\ge 1$, \eqref{eq:uniformGromov5} implies 
    \begin{align*}
        |\p^\ga F_a^{(\al,\be)}(u,0)w^{\be-\ta}| 
        \le L\, \frac{\om(d(u,\p D_a))}{d(u,\p D_a)^{|\ga|}} |w|^{|\ga|} 
        \le L\, \om(d(u,\p D_a)),
    \end{align*}
    and, if $\ga = 0$, \eqref{eq:uniformW2} gives
    \begin{align*}
        |\p^\ga F_a^{(\al,\be)}(u,0)w^{\be-\ta}| 
        = 
        |F_a^{(\al,\be)}(u,0)| \lesssim \om(d(u,\p T_a)) =  \om(d(u,\p D_a)),
    \end{align*}
    since $(F_a)_{a \in A}$ is flat on $(\p T_a \times 0)_{a \in A}$.

    To prove the statement for $|\ka|<m$, we proceed by induction on $m - |\ka|$.
    Suppose that the assertion is already shown for every $\la \in \N^n$ with $|\ka| < |\la|\le m$.
    Since the open line segment $(x,x')$ connecting $x=(u,w)$ and $x'=(u',0)$ 
    is contained in $\De(D_a \times 0)$,
    we have, by induction hypothesis, where $x''=(u'',w'')$,
    \begin{align*}
        |\p^\ka h_a(u,w)| 
        &\le 
        \sup_{x'' \in (x,x')} \sum_{j=1}^n |\p^{\ka+(j)} h_a(u'',w'')| |x-x'|
        \\
        &\lesssim 
        \sup_{x'' \in (x,x')} \om(d(u'',\p D_a)) d(u'',\p D_a)^{m-|\ka|-1}  \, d(u,\p D_a)
        \\
        &\lesssim 
        \om(d(u,\p D_a)) d(u,\p D_a)^{m-|\ka|},
    \end{align*}
    since $d(u'',\p D_a) \le d(u,\p D_a)$.
\end{proof}

It follows from \Cref{lem:uniformStep1.2} and \Cref{lem:uniformStep1.1}  
that, for each $a \in A$,
\begin{equation} \label{eq:adapted}
    |\p^\ka h_a(u,w)| \le C\, \om(\min\{1,d(u,\p D_{a})\}) \min\{1,d(u,\p D_{a})\}^{m-|\ka|}, 
\end{equation}
for all $(u,w) \in \De(D_{a} \times 0)$, all 
$\ka \in \N^n$, $|\ka| \le m$, 
and a constant $C>0$ independent of $a \in A$. 
Indeed, by \Cref{lem:uniformStep1.2}, 
\begin{equation*}
    |\p^\ka h_a(u,w)| \le C = \frac{C}{\om(1)} \cdot \om(1) 1^{m-|\ka|}, \quad |\ka|\le m, 
\end{equation*}
for all $(u,w) \in \De(D_a \times 0)$, which, together with \eqref{eq:uniformStep1.1}, gives \eqref{eq:adapted}.

Now \Cref{prop:uniformPawlucki39} (see also \Cref{rem:uniformPawlucki39})
implies that the family $(f_a)_{a \in A}$ is bounded in $C^{m,\om}(\R^n)$.
Indeed, \Cref{lem:uniformStep1.2}, \Cref{lem:uniformStep1.1},  and \eqref{eq:adapted} guarantee that the assumptions of 
\Cref{prop:uniformPawlucki39} are satisfied, where $t_a(u) = \min\{1, d(u,\p D_{a})\}$. 
Condition \eqref{eq:uniformr} holds thanks to \eqref{eq:rt0} and \Cref{rem:Krt}. 
We also get 
\begin{equation} \label{eq:Step1add}
    |\p^\ka f_a(u,w)| \le C\, \om(d(u,\p D_{a})) d(u,\p D_{a})^{m-|\ka|}, 
\end{equation}
for all $(u,w) \in \De(D_{a} \times 0)$, all 
$\ka \in \N^n$, $|\ka| \le m$, 
and a constant $C>0$ independent of $a \in A$.

\subsection*{Step 2} 

For all $a \in A$, $E'_a = \ol S_a = \ol T_a \times 0$, but possibly $E'_a$ 
is a proper subset of $E_a$ for some $a \in A$. 
Consider the definable family $(r_a)_{a \in A}$ of functions $r_a : T_a \to (0,\infty)$ given by
\[
    r_a(u) := 
    \begin{cases}
        \inf \{|w| : (u,w) \in E_a \setminus S_a\} & \text{ if } \{w : (u,w) \in E_a \setminus S_a\} \ne \emptyset,
        \\
        1 & \text{ otherwise.}
    \end{cases}
\]
Since $F_a$ is flat on $E_a \setminus S_a$ we have (by \eqref{eq:uniformW1} and \eqref{eq:uniformW2}) 
\begin{equation} \label{eq:Step2flat}
    |F^\ka_a(u,0)| \le C\, \om(r_a(u)) r_a(u)^{m-|\ka|}
\end{equation}
for all $u \in T_a$, all $\ka \in \N^n$, $|\ka| \le m$, and a constant $C>0$ independent of $a \in A$.
(In the case that $\{w : (u,w) \in E_a \setminus S_a\} = \emptyset$, it follows from \eqref{eq:uniformW1} and 
we have to replace $C$ by $C/\om(1)$.)

By \Cref{thm:uniformLapstrat} and \Cref{prop:forr}, 
there is a uniform $\La_p$-stratification of $(\ol T_a)_{a \in A}$ such that 
\[
    \ol T_a = Q_{a,1} \cup \cdots \cup Q_{a,s} \cup Z_a,
\]
where, for each $a \in A$ and each $i=1,\ldots,s$, $Z_a$ is closed with $\dim Z_a <k$, each $Q_{a,i}$ 
is an open $k$-dimensional $\La_p$-cell with constant independent of $a \in A$, 
$r_a$ is $C^p$ on $Q_{a,i}$, 
and either 
\begin{itemize}
    \item[(i)]  $|\p_j r_a| \le 1$, for each $j=1,\ldots,k$, on $Q_{a,i}$, 
        in which case we may assume that $|\p^\al r_a(u)| d(u,\p Q_{a,i})^{|\al|-1}$, $1 \le |\al|\le p$, is bounded on $Q_{a,i}$ 
        by a constant independent of $a \in A$,
        by \Cref{cor:La1top2},
        or
    \item[(ii)] $|\p_j r_a(u)| >1$ for some $j$ on $Q_{a,i}$.
\end{itemize}

By ($\mathbf{I}_{k-1}$),
we may assume that $(F_a)_{a \in A}$ is flat on $(Z_a \times 0)_{a \in A}$ and 
hence on $(\p Q_{a,i} \times 0)_{a \in A}$ for each $i=1,\ldots,s$.

Now it is enough to show that, for every $i=1,\ldots,s$, $(F_a|_{E_a \cap (\ol Q_{a,i} \times \R^\ell)})_{a \in A}$
admits a definable bounded family $(f_{a,i})_{a \in A}$ of $C^{m,\om}$-extensions to $\R^n$ 
that is $m$-flat outside $(\De(Q_{a,i} \times 0))_{a \in A}$ 
and $C^p$ outside $(\ol Q_{a,i} \times 0)_{a \in A}$.
To this end, we fix $i$ and drop it from the notation.

Step 1 gives a definable bounded family $(g_a)_{a \in A}$ of $C^{m,\om}$-extensions 
to $\R^n$ of $(F_a|_{\ol Q_{a} \times 0})_{a \in A}$
that is $m$-flat outside $(\De(Q_a \times 0)_{a \in A}$ 
and $C^p$ outside $(\ol Q_a \times 0)_{a \in A}$.
By Taylor's formula and \eqref{eq:Step2flat}, for each $a \in A$,
\begin{equation} \label{eq:Step22}
    |\p^\ka g_a(u,w)| \le C \, \om(r_a(u))\, r_a(u)^{m -|\ka|} 
\end{equation}
for all $(u,w) \in Q_{a} \times \R^\ell$, $|w| < C' \, r_a(u)$, and all $\ka \in \N^n$, $|\ka| \le m$,
where $C,C'>0$ are independent of $a \in A$.
Similarly, we have
\begin{equation} \label{eq:Step2add}
    |\p^\ka g_a(u,w)| \le C \, \om(d(u,\p Q_a))\, d(u,\p Q_a)^{m -|\ka|} 
\end{equation}
for all $(u,w) \in Q_a \times \R^\ell$, 
$|w| < C' \, d(u,\p Q_a)$, and all $\ka \in \N^n$, $|\ka| \le m$,
where $C,C'>0$ are independent of $a \in A$.

In Case (ii), one can easily see (see \cite[p. 94]{Kurdyka:2014aa}) that, for each $a \in A$, 
$r_a(u) \ge d(u,\p Q_a)$ for $u \in Q_a$, 
so that $E_a \setminus S_a \subseteq \R^n \setminus \De(Q_a \times 0)$.
That means that $g_a$ is a $C^{m,\om}$-extension to $\R^n$ of $F_a|_{E_a \cap (\ol Q_a \times 0)}$, 
and we are done.

In Case (i), a modification is necessary: we define, for each $a \in A$, 
\[
    f_a(u,w) := 
    \begin{cases}
        \prod_{i=1}^\ell \xi\Big(\sqrt{\ell} \frac{w_i}{r_a(u)}\Big) \cdot g_a(u,w) &\text{ if } u \in Q_a,
        \\
        0 &\text{ otherwise,}
    \end{cases}
\]
where $\xi : \R \to \R$ is a semialgebraic $C^p$-function that is $1$ near $0$ and vanishes outside $(-1,1)$.
Note that $(f_a)_{a \in A}$ is a definable family of functions $f_a : \R^n \to \R$.
Moreover, we set
\[
    \De'(Q_{a} \times 0) := \{(u,w) \in Q_{a} \times \R^\ell : |w| < t_a(u) \}
\]
with
\[
    t_a(u):= \min\{r_a(u), d(u,\p Q_{a})\},
\]
and claim that, for each $a \in A$, $f_a$ is of class $C^{m,\om}$ on $\De'(Q_a \times 0)$ with $C^{m,\om}$-norm bounded by a constant independent of $a \in A$, 
and
\begin{equation} \label{eq:Step23}
    |\p^\ka f_a(u,w)| \le C \, \om(t_a(u))\, t_a(u)^{m -|\ka|} 
\end{equation}
for $(u,w) \in \De'(Q_a \times 0)$ and all $\ka \in \N^n$, $|\ka|\le m$,
where $C>0$ is independent of $a \in A$.

To see this, let us first assume that $r_a(u) < d(u,\p Q_a)$ so that $t_a(u) =r_a(u)$. 
Since we are in Case (i), we find that, thanks to \eqref{eq:inv}, 
\[
    \Big|\p^\al \Big(\frac{1}{r_a} \Big)(u) \Big| \le C\, r_a(u)^{-|\al|-1},\quad u \in Q_a, \, |\al|\le p,
\]
for a constant $C>0$ independent of $a \in A$.
Thus, 
the claim follows from \eqref{eq:Step22} and \Cref{prop:uniformPawlucki39}.

If $r_a(u) \ge d(u,\p Q_a)$ (that is, $t_a(u) = d(u,\p Q_a)$),
then similarly  
\[
    \Big|\p^\al \Big( \frac{1}{r_a} \Big)(u) \Big| \le C\, d(u,\p Q_a)^{-|\al|-1},\quad u \in Q_a, \, |\al|\le p,
\]
Then we infer the claim  from \eqref{eq:Step2add} and \Cref{prop:uniformPawlucki39}.

We conclude that $(f_a)_{a \in A}$ is the required family of definable bounded $C^{m,\om}$-extensions to 
$\R^n$ of 
$(F_a|_{E_a \cap (\ol Q_{a} \times \R^\ell)})_{a \in A}$ 
that is $m$-flat outside $(\De(Q_a \times 0))_{a \in A}$ 
and $C^p$ outside $(\ol Q_a \times 0)_{a \in A}$.
This ends Step 2.

\subsection*{Step 3}

The general case of \Cref{prop:afterred}:
for all $a \in A$,
$S_a = \Ga(\vh_a)$, $E_a' = \ol S_a \subseteq E_a$, where  
$\vh_a : T_a \to \R^\ell$ is not necessarily identically $0$. 
Consider the definable family $(s_a)_{a \in A}$ of functions $s_a : S_a \to (0,\infty)$ given by 
\[
    s_a(x) := \min\{d(x,E_a \setminus S_a), d(x,\p S_a)\}, \quad x \in S_a.
\]

For each $a \in A$, let $\ol \vh_a : \ol T_a \to \R^\ell$ be the continuous extension of $\vh_a$; see \Cref{ssec:regmap}.
Furthermore, we consider the maps
\begin{align*} 
    \vh_{a,\pm} &:  T_a \times \R^\ell \to T_a \times \R^\ell, \quad (u,w) \mapsto (u,w \pm \vh_a(u))
    \intertext{and}
    \ol \vh_{a,\pm} &: \ol T_a \times \R^\ell \to \ol T_a \times \R^\ell, \quad (u,w) \mapsto (u,w \pm \ol \vh_a(u)).
\end{align*}
Note that $\ol \vh_{a,+}$ is a bi-Lipschitz homeomorphism
with inverse $\ol \vh_{a,-}$ and Lipschitz constants independent of $a \in A$.

Since $F_a$ is flat on $E_a \setminus S_a$ and on $\p S_a$, we have (by \eqref{eq:uniformW2})
\begin{equation} \label{eq:Step3flat}
    |F_a^\ka(x)| \le C \, \om(s_a(x)) s_a(x)^{m-|\ka|}
\end{equation}
for all $x \in S_a$, all $\ka \in \N^n$, $|\ka|\le m$, and a constant $C>0$ independent of $a \in A$.
Setting 
\[
    t_a(u) := s_a(u,\vh_a(u)), \quad u \in T_a, 
\]
we have 
\begin{equation} \label{eq:Step3flatu}
    |F_a^\ka(u,\vh_a(u))| \le C \, \om(t_a(u)) t_a(u)^{m-|\ka|}, \quad u \in T_a, \, |\ka|\le m.
\end{equation}
The uniformity of the constants in the definition of $T_a$ and $\vh_a$ 
implies that $(\ol S_a)_{a \in A}$ and $(\p T_a \times \R^\ell)_{a \in A}$ are 
uniformly 
separated.
Observe that (by the definition of $s_a$) 
$t_a(u) \le C' \, d(u,\p T_a)$ for $C'>0$ independent of $a \in A$, since $\ol \vh_{a,+}$ is a bi-Lipschitz homeomorphism with 
Lipschitz constants independent of $a \in A$.

Thus \Cref{prop:uniformpullback} (and \Cref{lem:contexofjets}) implies 
that $(G_a)_{a \in A}$, where $G_a:=\vh_{a,+}^* (F_a|_{S_a})$, is a definable bounded family of Whitney jets 
of class $C^{m,\om}$ on $(T_a \times 0)_{a \in A}$
and extends to a definable bounded family of Whitney jets 
of class $C^{m,\om}$ on $(\ol T_a \times 0)_{a \in A}$
which is flat on $(\p T_a \times 0)_{a \in A}$ and such that 
\begin{equation} \label{eq:Step3flatG}
    |G_a^\ka(u,0)| \le C \, \om(t_a(u)) t_a(u)^{m-|\ka|}
\end{equation}
for all $u \in T_a$, all $\ka \in \N^n$, $|\ka|\le m$, and a constant $C>0$ independent of $a \in A$.
For each $a \in A$,
set $\tilde E_a := \ol \vh_{a,-}(E_a \cap (\ol T_a \times \R^\ell))$. 
Since $\ol \vh_{a,+}$ is a bi-Lipschitz homeomorphism with constants independent of $a \in A$, 
we may conclude
\begin{equation} \label{eq:Step3flatG2}
    |G_a^\ka(u,0)| \le C \, \om(d((u,0), \tilde E_a \setminus (T_a \times 0))) d((u,0), \tilde E_a \setminus (T_a \times 0))^{m-|\ka|}
\end{equation}
for all $u \in T_a$, all $\ka \in \N^n$, $|\ka|\le m$, and a constant $C>0$ independent of $a \in A$.
Thus $(\tilde G_a)_{a \in A}$, where 
\[
    \tilde G_a(u,w) := 
    \begin{cases}
        G_a(u,0) & \text{ if } (u,w) \in \ol T_a \times 0,
        \\
        0 & \text{ if } (u,w) \in \tilde E_a \setminus (\ol T_a \times 0),
    \end{cases}
\]
is a definable bounded family of Whitney jets of class $C^{m,\om}$ on $(\tilde E_a)_{a \in E}$ 
that is flat on $(\tilde E_a \setminus (T_a \times 0))_{a \in A}$.

By Step 2, 
there exists a definable bounded family $(\tilde g_a)_{a \in A}$ of $C^{m,\om}$-extensions 
to $\R^n$ of $(\tilde G_a)_{a \in A}$
that is $m$-flat on $(\tilde E_a \setminus (T_a \times 0))_{a \in A}$ as well as outside $(T_a \times \R^\ell)_{a \in A}$ and
$C^p$ outside $(\ol T_a \times 0)_{a \in A}$.

For each $a \in A$, define $f_a : \R^n \to \R$ by 
\[
    f_a(u,w) := 
    \begin{cases}
        (\tilde g_a\o \vh_{a,-})(u,w)  & \text{ if } (u,w) \in T_a \times \R^\ell,
        \\
        0 & \text{ otherwise.} 
    \end{cases}
\]
Then $(f_a)_{a \in A}$ is a definable bounded family of $C^{m,\om}$-extensions to $\R^n$ of $(F_a)_{a \in A}$
that is $C^p$ outside $(\ol S_a)_{a \in A}$, 
which follows again from \Cref{prop:uniformpullback} (with $M_a = T_a \times \R^\ell$ and $U_a = T_a$).

This completes the proof of \Cref{prop:afterred}, hence of ($\mathbf{I}_{k}$), and thus the proof of \Cref{thm:mainuniform}.

\section{Further applications} 
\label{sec:applications}

We present a local version of \Cref{thm:mainuniform}, 
we discuss the dependence of the bounded extension on the modulus of continuity which leads to the proof of \Cref{thm:Cmbd}, 
and finally 
we obtain a definable version of a correspondence between Whitney jets of class $C^{m,\om}$ and certain Lipschitz maps,
which was first observed by Shvartsman \cite{Shvartsman:2008aa}. 

\subsection{Definable $C^{m,\om}_{\on{loc}}$-extensions} \label{ssec:local}

Let $U \subseteq \R^n$ be open.
We denote by $C^{m,\om}_{\on{loc}}(U)$ the space of functions $f : U \to \R$ 
such that $f|_V \in C^{m,\om}(V)$, for all relatively compact open subsets $V \subseteq U$.

Let $E \subseteq \R^n$ be a closed set.
An $m$-jet $F$ on $E$ is called a \emph{(definable) Whitney jet of class $C^{m,\om}_{\on{loc}}$ on $E$}
if $F|_K$ is a (definable) Whitney jet of class $C^{m,\om}$ on $K$, for all (definable) compact subsets $K \subseteq E$.
A $C^{m,\om}_{\on{loc}}$-function $f : \R^n \to \R$ is a \emph{$C^{m,\om}_{\on{loc}}$-extension to $\R^n$ of $F$} 
if $J_E^m(f) = F$.

Let $(E_a)_{a \in A}$ be a family of closed sets $E_a \subseteq \R^n$.
A family $(F_a)_{a \in A}$ of Whitney jets of class $C^{m,\om}_{\on{loc}}$ on $E_a$
is called a \emph{(definable) bounded family of Whitney jets of class $C^{m,\om}_{\on{loc}}$}
if $(F_a|_{K_a})_{a \in A}$ is a (definable) bounded family of Whitney jets of class $C^{m,\om}$ 
for each (definable) subfamily $(K_a)_{a \in A}$ of $(E_a)_{a \in A}$ consisting of 
(definable) compact sets $K_a \subseteq E_a$.

A family $(f_a)_{a \in A}$ of $C^{m,\om}_{\on{loc}}$-functions $f_a : \R^n \to \R$ 
is called a \emph{(definable) bounded family of $C^{m,\om}_{\on{loc}}$-extensions to $\R^n$ of $(F_a)_{a \in A}$} 
if $f_a$ is a $C^{m,\om}_{\on{loc}}$-extension to $\R^n$ of $F_a$, for each $a \in A$, 
and, for each (definable) relatively compact subset $V \subseteq \R^n$, 
$(f_a|_V)_{a \in A}$ is a (definable) bounded family of $C^{m,\om}$-functions.

\begin{corollary} \label{cor:local}
    Let $0 \le m\le p$ be integers. Let $\om$ be a modulus of continuity.
    Let $(E_a)_{a \in A}$ be a definable family of closed subsets $E_a$ of $\R^n$.
    For any definable bounded family $(F_a)_{a \in A}$ of Whitney jets of class $C^{m,\om}_{\on{loc}}$ on $(E_a)_{a \in A}$
    there exists a definable bounded family $(f_a)_{a \in A}$ of $C^{m,\om}_{\on{loc}}$-extensions to $\R^n$ 
    of $(F_a)_{a \in A}$ that is $C^p$ outside $(E_a)_{a \in A}$.
\end{corollary}


\begin{proof}
   For integers $k \ge 1$, consider the definable sets $U_k := \{x \in \R^n : k-2 < |x| < k\}$; note that $U_1$ is the unit ball, 
   $U_2$ is a punctured ball, 
   and $U_k$, for $k\ge 3$, are annuli centered at the origin.
   The sets $U_k$, for $k \ge 1$, form an open cover of $\R^n$ with the property that $U_k \cap U_\ell \ne \emptyset$ if 
   and only if $|k-\ell|\le 1$.
   Fix an integer $p \ge m+1$. 
   There exists a partition of unity $\{\vh_k\}_{k\ge 1}$ of class $C^p$ subordinated to the cover $\{U_k\}_{k \ge 1}$, where each $\vh_k$ is definable:  
   $\vh_k \in C^p(\R^n)$, $\vh_k \ge 0$, $\on{supp} \vh_k \subseteq U_k$, for all $k\ge 1$, 
   the family $\{\on{supp} \vh_k\}_{k \ge 1}$ is locally finite,
   and $\sum_{k \ge 1} \vh_k=1$. 
   For instance, let $h : \R \to \R$ be a nonnegative definable function of class $C^p$ 
   such that $\on{supp} h = [-3/4,3/4]$
   and set $\ps_1(x) := h(|x|^2)$ and
   $\ps_k(x) := h(|x| - (k-1))$, for $k \ge 2$. Then $\ps := \sum_{k \ge 1} \ps_k$ is of class $C^p$ and everywhere positive 
   (locally it is a finite sum).
   Thus $\vh_k := \ps_k/\ps$ is as required; it is definable, since in a neighborhood of $\on{supp} \vh_k = \on{supp} \ps_k$ the denominator $\ps$ is represented by a finite sum of 
   definable functions.

   Let $(F_a)_{a \in A}$ be a definable bounded family of Whitney jets of class $C^{m,\om}_{\on{loc}}$ on $(E_a)_{a \in A}$.
   For each $k \ge 1$, $(F_a|_{\ol U_k})_{a \in A}$ is a definable bounded family of 
   Whitney jets of class $C^{m,\om}$ on $(E_a \cap \ol U_k)_{a \in A}$.
   By \Cref{thm:mainuniform}, 
   there exists a definable bounded family $(f^k_a)_{a \in A}$ of 
   $C^{m,\om}$-extensions to $\R^n$ of $(F_a|_{\ol U_k})_{a \in A}$ 
   such that $f^k_a$ is of class $C^p$ outside $E_a \cap \ol U_k$ for all $a \in A$; 
   if $E_a \cap \ol U_k = \emptyset$ we set $f^k_a := 0$.
   Let $f_a := \sum_{k=1}^\infty \vh_k f_a^k$, for $a \in A$.
   The function $f_a$ is of class $C^{m,\om}_{\on{loc}}$ on $\R^n$ and $C^p$ outside $E_a$, 
   since the defining sum is finite on every compact set and $p \ge m+1$.
   Let $x \in E_a$. There exist a neighborhood $U$ of $x$ and $k \ge 1$ such that $U \subseteq U_{k} \cup U_{k+1} \cup U_{k+2}$ 
   and $U \cap \bigcup_{\ell \not\in \{k,k+1,k+2\} } U_\ell = \emptyset$.
   So, for each $\ka \in \N^n$ with $|\ka| \le m$,
   \begin{align*}
       \p^\ka f_a(x)  &= \sum_{i=0}^2 \sum_{\si \le \ka} \binom{\ka}{\si} \p^\si \vh_{k+i}(x) \p^{\ka-\si} f_a^{k+i}(x)
       \\
       &= \sum_{i=0}^2 \sum_{\si \le \ka} \binom{\ka}{\si} \p^\si \vh_{k+i}(x)  F_a^{\ka-\si}(x)
       \\
       &=  \sum_{\si \le \ka} \binom{\ka}{\si} \p^\si\Big(\sum_{i=0}^2  \vh_{k+i}(x)\Big)  F_a^{\ka-\si}(x)
       \\
       &= \sum_{\si \le \ka} \binom{\ka}{\si} \p^\si (1)  F_a^{\ka-\si}(x)
       \\
       &= F^\ka_a(x).
   \end{align*}
   Thus, $f_a$ is a $C^{m,\om}_{\on{loc}}$-extension to $\R^n$ of $F_a$.

   Fix a definable relatively compact subset $V \subseteq \R^n$. 
   There exists $K \in \N$ such that $\ol V \cap U_k= \emptyset$ for all $k > K$. 
   In particular, $f_a(x) := \sum_{k=1}^K \vh_k(x) f_a^k(x)$, for all $x \in V$ and $a \in A$.
   Hence, $(f_a|_V)_{a \in A}$ is a definable bounded family of $C^{m,\om}$-functions.
\end{proof}

\begin{remark} \label{rem:globaldef}
    We do not say that $f_a$ is definable as a global function $f_a : \R^n \to \R$, 
    because the gluing argument (based on the partition of unity) involves an infinite sum.
\end{remark}

\subsection{Dependence on the modulus of continuity}
\label{sec:omdep}

The main result, \Cref{thm:mainuniform}, only depends in a weak sense on the modulus of continuity $\om$, 
namely, 
the uniform constant $C$ occasionally must be multiplied by $\om(1)$ or by $\om(1)^{-1}$;
see \eqref{eq:omub}, \eqref{eq:divom1}, \eqref{eq:adapted}, and \eqref{eq:Step2flat}.

Thus, we can allow in \Cref{thm:mainuniform} that, for each $a \in A$, 
$F_a$ is a Whitney jet of class $C^{m,\om_a}$ on $E_a$, where $\om_a$ is a modulus of continuity 
and there is a constant $C>0$ independent of $a \in A$ 
such that 
\begin{equation} \label{eq:uniformom}
    C^{-1} \le \om_a(1) \le C,\quad a \in A.
\end{equation}
Then the statement is the following:

\begin{theorem} \label{thm:omdep}
    Let $0 \le m\le p$ be integers. Let $(\om_a)_{a \in A}$ be a family of moduli of continuity satisfying
    \eqref{eq:uniformom}.
    Let $(E_a)_{a \in A}$ be a definable family of closed subsets $E_a$ of $\R^n$.
    For any definable family $(F_a)_{a \in A}$ of Whitney jets $F_a$ of class $C^{m,\om_a}$ on $E_a$ such that 
    \begin{align}
        &\sup_{a \in A} \sup_{x \in E_a} \sup_{|\ga|\le m} |F_a^\ga(x)| < \infty,
        \intertext{and}
        &\sup_{a \in A} \sup_{x\ne y \in E_a} \sup_{|\ga|\le m} \frac{|(R^m_x F_a)^\ga (y)|}{\om_a(|x-y|) |x-y|^{m-|\ga|}} < \infty,
    \end{align}
    there exists a definable family $(f_a)_{a \in A}$ of $C^{m,\om_a}$-extensions $f_a$ to $\R^n$ of $F_a$
    such that $f_a$ is of class $C^p$ outside $E_a$, for all $a \in A$, and 
    \begin{equation}
        \sup_{a \in A} \|f_a\|_{C^{m,\om_a}(\R^n)} < \infty.
    \end{equation}
\end{theorem}

\subsection{Proof of \Cref{thm:Cmbd}} \label{sec:proofCm}

Let $(F_a)_{a \in A}$ be a definable family of Whitney jets of class $C^{m}$ on $(E_a)_{a \in A}$, 
where $E_a \subseteq \R^n$ is compact.
We say that the family $(F_a)_{a \in A}$ 
is \emph{bounded} if 
\begin{align}
    &\sup_{a \in A} \sup_{x \in E_a} \sup_{|\ga|\le m} |F^\ga_a(x)| < \infty 
    \intertext{and}
    &\sup_{a \in A} \sup_{x\ne y \in E_a} \sup_{|\ga|\le m} \frac{|(R^m_x F_a)^\ga (y)|}{|x-y|^{m-|\ga|}} < \infty.
    \label{eq:Cmbd}
\end{align}

\begin{proof}[Proof of \Cref{thm:Cmbd}]
    We modify slightly an argument used in \cite[Proposition IV.1.5]{Tougeron72}.
    For each $a \in A$, consider 
    \[
        \si_a(t) := \sup_{\substack{x\ne y \in E_a\\ |x-y|\le t}} \sup_{|\ga|\le m} \frac{|(R^m_x F_a)^\ga (y)|}{|x-y|^{m-|\ga|}}, \quad t>0, \quad \si_a(0):=0.
    \]
    Then $\si_a : [0,\infty) \to [0,\infty)$ is an increasing function that is continuous at $0$ 
    and 
    \begin{equation*}
        \si_a(t) = \si_a(\on{diam} E_a), \quad t \ge \on{diam} E_a. 
    \end{equation*}
    Thus also $\ta_a :[0,\infty) \to [0,\infty)$, defined by 
    \[
        \ta_a(t) := 
        \begin{cases}
            \si_a(t) &\text{ if } t<1,
            \\
            \max\{1,\si_a(t)\} &\text{ if } t\ge 1,
        \end{cases}
    \]
    is increasing and continuous at $0$ with 
    \begin{equation} \label{eq:taafinite}
        \ta_a(t) \le \max\{1, \si_a(\on{diam} E_a) \}, \quad t \ge 0. 
    \end{equation}
    Let $\om_a$ be the least concave majorant of $\ta_a$ 
    which is finite, thanks to \eqref{eq:taafinite}. 
    Then $\om_a$ is a modulus of continuity 
    and 
    \[
        \sup_{a \in A} \sup_{x\ne y \in E_a} \sup_{|\ga|\le m} \frac{|(R^m_x F_a)^\ga (y)|}{\om_a(|x-y|)|x-y|^{m-|\ga|}} \le 1.
    \]
    Moreover, $\om_a(1) \ge 1$ and, by \eqref{eq:taafinite}, 
    \begin{equation*}
        \om_a(t)  \le \max\{1,\si_a(\on{diam} E_a)\} \le C, \quad t \ge 0,
    \end{equation*}
    for a constant $C>0$ independent of $a \in A$, thanks to  \eqref{eq:Cmbd}.
    In particular, \eqref{eq:uniformom} is satisfied.

    Thus \Cref{thm:omdep} implies 
    that there is a definable family $(f_a)_{a \in A}$ 
    such that, for each $a \in A$, $f_a$ is a $C^{m,\om_a}$-extension to $\R^n$ of $F_a$, $C^p$ outside $E_a$, 
    and 
    \[
        \sup_{a \in A} \|f_a\|_{C^{m,\om_a}(\R^n)} < \infty.
    \]
    In particular, $(f_a)_{a \in A}$ is a bounded family of $C^m$-functions.
\end{proof}

\subsection{Definable Whitney jets as Lipschitz maps} \label{ssec:corrLip}

We end with a few observations on a definable version of a correspondence, 
due to Shvartsman \cite{Shvartsman:2008aa}, between Whitney jets of class 
$C^{m,\om}$ and certain Lipschitz maps. Here the notation follows closely the one of \cite{Shvartsman:2008aa}.

Let $\om$ be a modulus of continuity and $m$ a positive integer.
For $\al \in \N^n$ with $|\al|<m$ let $\ps_\al$ be the inverse of the (strictly increasing) function 
$s \mapsto s^{m-|\al|} \om(s)$ and put $\vh_\al := \om\o \ps_\al$.
For $|\al|=m$, set $\vh_\al(t) := t$.

Let $\cP_m$ denote the space of real polynomials of degree at most $m$ in $n$ variables.
For $T_i=(P_i,x_i) \in \cP_m \times \R^n$, $i=1,2$, 
define 
\begin{align*}
    \de_\om(T_1,T_2) := \max\Big\{\om(|x_1-x_2|), 
    \max_{\substack{|\al|\le m\\ i =1,2}} \vh_\al(|\p^\al(P_1-P_2)(x_i)|)\Big\}.
\end{align*}
Then we get a metric $d_\om$ on $\cP_m \times \R^n$ by setting
\[
    d_\om(T,T') := \inf \sum_{j=0}^{k-1} \de_\om(T_j,T_{j+1}),
\]
where the infimum is taken over all finite sequences $T=T_1,T_2,\ldots,T_k=T'$ in $\cP_m \times \R^n$. 
It turns out (see \cite[Theorem 2.1]{Shvartsman:2008aa}) that 
\[
    d_\om((P,x),(P',x')) \le \de_\om((P,x),(P',x')) \le d_\om((e^n P,x),(e^n P',x')).
\]
Let $\cT_{m,n}$ be the metric space $(\cP_m \times \R^n,d_\om)$. 
For a nonempty subset $X \subseteq \R^n$, we denote by $X_\om$ the metric space $(X,(x,y) \mapsto \om(|x-y|))$.
Let $\mathbf{Lip}(X_\om,\cT_{m,n})$ be the space of Lipschitz maps $T: x \mapsto (P_x,z_x)$ 
such that $\max_{|\al|\le m} \sup_{x \in X} |\p^\al P_x(x)| < \infty$, equipped with the norm
\begin{align*}
    \|T\|^*_{\on{LO}(X)} &:= \max_{|\al|\le m} \sup_{x \in X} |\p^\al P_x(x)| 
    \\
                    &\quad + \inf\{\la >0 : d_\om(\la^{-1}T(x),\la^{-1}T(y)) \le \om(|x-y|) 
                    \text{ for all } x,y \in X\},
\end{align*}
where $\la^{-1}T(x) := (\la^{-1}P_x,z_x)$.
Let $T^m_x f$ be the Taylor polynomial of order $m$ at $x$ of a $C^m$-function $f$.

Now let us recall a result of \cite{Shvartsman:2008aa}.

\begin{proposition}[{\cite[Propositions 1.9 and 2.8]{Shvartsman:2008aa}}]
    Let $X \subseteq \R^n$ be a closed set. 
    Given a family of polynomials $\{P_x \in \cP_m : x \in X\}$,
    there exists $f \in C^{m,\om}(\R^n)$ such that $T^m_x f = P_x$ for all $x \in X$
    if and only if the map $T : x \mapsto (P_x,x)$ belongs to $\mathbf{Lip}(X_\om,\cT_{m,n})$.
    We have 
    \begin{equation*}
        \inf\{\|f\|_{C^{m,\om}(\R^n)} : T^m_x f = P_x \text{ for all } x\in X \} \approx \|T\|^*_{\on{LO}(X)}
    \end{equation*}
    in the sense that either side is bounded by a constant $C(m,n)$ times the other side.  
    If, moreover, $T : x \mapsto (P_x,x)$ belongs to $\mathbf{Lip}(X_\om,\cT_{m,n})$, 
    then $T$ has an extension $\widetilde T : x \mapsto (\widetilde P_x,x)$ 
    in $\mathbf{Lip}(\R^n_\om,\cT_{m,n})$ satisfying
    \begin{equation*}
        \|\widetilde T\|^*_{\on{LO}(\R^n)} \le C(m,n)\, \|T\|^*_{\on{LO}(X)}.
    \end{equation*}
\end{proposition}

These results are based on the classical extension theorem for Whitney jets of class $C^{m,\om}$. 
As a consequence of \Cref{thm:main}, we may conclude the following definable version, where, 
provided that $X$ is definable,
$\mathbf{Lip}_{\on{def}}(X_\om,\cT_{m,n})$ is the subspace of definable maps $T : x \mapsto (P_x,z_x)$ in 
$\mathbf{Lip}(X_\om,\cT_{m,n})$, which means that $z_x$ and the coefficients of $P_x$ are definable maps in $x$.
Recall that $C^{m,\om}_{\on{def}}(\R^n)$ is the subspace of $C^{m,\om}(\R^n)$ consisting of all definable functions in   
$C^{m,\om}(\R^n)$.

\begin{proposition} \label{prop:Lip}
    Let $X \subseteq \R^n$ be a definable closed set. 
    Given a definable family of polynomials $\{P_x \in \cP_m : x \in X\}$,
    there exists $f \in C^{m,\om}_{\on{def}}(\R^n)$ such that $T^m_x f = P_x$ for all $x \in X$
    if and only if the map $T : x \mapsto (P_x,x)$ belongs to $\mathbf{Lip}_{\on{def}}(X_\om,\cT_{m,n})$.
    If, moreover, $T : x \mapsto (P_x,x)$ belongs to $\mathbf{Lip}_{\on{def}}(X_\om,\cT_{m,n})$, 
    then $T$ has an extension $\widetilde T : x \mapsto (\widetilde P_x,x)$ 
    in $\mathbf{Lip}_{\on{def}}(\R^n_\om,\cT_{m,n})$.
\end{proposition}

Concerning the existence of uniform bounds for the norms, 
remarks similar to the ones in \cite[Section 4.4]{ParusinskiRainer:2023ab} apply.
But \Cref{thm:mainuniform} implies the following supplement.

\begin{proposition}
    Suppose that in the setting of \Cref{prop:Lip},
    the family of polynomials depends definably on additional parameters $a \in A$, 
    i.e., a definable family of polynomials $\{P^a_x \in \cP_m : x \in X,\, a \in A\}$ is given. 
    Then there exists a bounded family $(f^a)_{a \in A}$ of definable $C^{m,\om}$-functions $f^a : \R^n \to \R$ such that 
    \[
        T_x^m f^a = P^a_x\quad \text{ for all } x \in X \text{ and } a \in A 
    \]
    if and only if $(T^a : x \mapsto (P^a_x,x))_{a \in A}$ forms a bounded subset of $\mathbf{Lip}_{\on{def}}(X_\om,\cT_{m,n})$.
    If, moreover,  $(T^a : x \mapsto (P^a_x,x))_{a \in A}$ forms a bounded subset of $\mathbf{Lip}_{\on{def}}(X_\om,\cT_{m,n})$, 
    then there is a family $(\widetilde T^a : x \mapsto (\widetilde P^a_x,x))_{a \in A}$ of extensions $\widetilde T^a$ of 
    $T^a$ which forms a bounded subset of 
    $\mathbf{Lip}_{\on{def}}(\R^n_\om,\cT_{m,n})$.
\end{proposition}

\subsection*{Acknowledgements}
This research was funded in whole or in part by the Austrian Science Fund (FWF) DOI 10.55776/P32905.
For open access purposes, the author has applied a CC BY public copyright license to any author-accepted manuscript version arising from this submission.
A large part of this work has been done at Mathematisches Forschungsinstitut Oberwolfach 
(\emph{Oberwolfach Research Fellows (OWRF) ID 2244p, October 31 - November 19, 2022}).
We are grateful for the support and the excellent working conditions.


\begin{thebibliography}{10}

\bibitem{Azagra:2018aa}
D.~Azagra, E.~Le Gruyer, and C.~Mudarra, 
\emph{{Explicit formulas for $C^{1,1}$ and $C^{1,\omega}_{\operatorname{conv}}$ extensions of 1-jets in Hilbert and superreflexive spaces}}, 
{J. Funct. Anal.} \textbf{274} (2018), no.~10, 3003--3032.
\url{https://doi.org/10.1016%2Fj.jfa.2017.12.007}

\bibitem{BartleGraves52}
R.G.~Bartle and L.M.~Graves, 
\emph{Mappings between function spaces}, Trans. Amer. Math. Soc. \textbf{72} (1952), 400--413.
\url{https://doi.org/10.2307/1990709}

\bibitem{BrudnyiBrudnyi12Vol1}
A.~Brudnyi and Y.~Brudnyi, 
\emph{Methods of geometric analysis in extension and trace problems. {V}olume 1}, 
Monographs in Mathematics, vol. 102, Birkh{\"a}user/Springer Basel AG, Basel, 2012. 

\bibitem{vandenDries98}
L.~van~den Dries, 
\emph{Tame topology and o-minimal structures}, 
London Mathematical Society Lecture Note Series, vol. 248, Cambridge University Press, Cambridge, 1998.

\bibitem{vandenDriesMiller96}
L.~van~den Dries and C.~Miller, 
\emph{Geometric categories and o-minimal structures}, 
Duke Math. J. \textbf{84} (1996), no.~2, 497--540.
\url{https://doi.org/10.1215/S0012-7094-96-08416-1}

\bibitem{Glaeser:1958aa}
G.~Glaeser, \emph{{\'{E}}tude de quelques alg{\`{e}}bres tayloriennes}, 
J. Analyse Math. \textbf{6} (1958), no.~1, 1--124.
\url{https://doi.org/10.1007%2Fbf02790231}

\bibitem{Gromov87}
M.~Gromov, \emph{{Entropy, homology and semialgebraic geometry [after Y. Yomdin]}}, S\'eminaire Bourbaki. Ast\'{e}risque (1987), no.~145-146, 5, 225--240.

\bibitem{KurdykaParusinski06}
K.~Kurdyka and A.~Parusi{\'n}ski, 
\emph{Quasi-convex decomposition in o-minimal structures. {A}pplication to the gradient conjecture}, 
Singularity theory and its applications, Adv. Stud. Pure Math., vol.~43, Math. Soc. Japan, Tokyo, 2006, pp.~137--177. 
\url{https://doi.org/10.2969/aspm/04310137}


\bibitem{Kurdyka:1997ab}
K.~Kurdyka and W.~Paw{\l}ucki, 
\emph{Subanalytic version of {Whitney}'s extension theorem}, 
Stud. Math. \textbf{124} (1997), no.~3, 269--280.
\url{https://doi.org/10.4064/sm-124-3-269-280}

\bibitem{Kurdyka:2014aa}
\bysame, 
\emph{O-minimal version of {Whitney}'s extension theorem}, 
Stud. Math. \textbf{224} (2014), no.~1, 81--96.
\url{https://doi.org/10.4064/sm224-1-4}

\bibitem{ParusinskiRainer:2023ab}
A.~{Parusi\'nski} and A.~{Rainer}, 
\emph{{Definable Lipschitz selections for affine-set valued maps}}, 
accepted for publication in Israel Journal of Mathematics,
arXiv:2306.09155. 
\url{https://arxiv.org/abs/2306.09155}

\bibitem{Pawlucki08aa}
W.~Paw{\l}ucki, 
\emph{{A linear extension operator for Whitney fields on closed o-minimal sets}}, 
Annales de l'Institut Fourier \textbf{58} (2008), no.~2, 383--404.
\url{https://doi.org/10.5802%2Faif.2355}

\bibitem{Shvartsman:2008aa}
P.~Shvartsman, 
\emph{{The Whitney extension problem and Lipschitz selections of set-valued mappings in jet-spaces}}, 
{Trans. Amer. Math. Soc.} \textbf{360} (2008), no.~10, 5529--5550.
\url{https://doi.org/10.1090%2Fs0002-9947-08-04469-3}

\bibitem{Thamrongthanyalak:2017aa}
A.~Thamrongthanyalak, 
\emph{Whitney's extension theorem in o-minimal structures}, 
Ann. Pol. Math. \textbf{119} (2017), no.~1, 49--67.
\url{https://doi.org/10.4064/ap3940-2-2017}

\bibitem{Tougeron72}
J.-C. Tougeron, 
\emph{Id\'eaux de fonctions diff\'erentiables}, 
Springer-Verlag, Berlin, 1972, Ergebnisse der Mathematik und ihrer Grenzgebiete, Band 71.

\bibitem{Whitney34a}
H.~Whitney, 
\emph{Analytic extensions of differentiable functions defined in closed sets}, 
Trans. Amer. Math. Soc. \textbf{36} (1934), no.~1, 63--89. 
\url{http://dx.doi.org/10.2307/1989708}

\end{thebibliography}

\def\cprime{$'$}
\providecommand{\bysame}{\leavevmode\hbox to3em{\hrulefill}\thinspace}
\providecommand{\MR}{\relax\ifhmode\unskip\space\fi MR }
\providecommand{\MRhref}[2]{%
  \href{http://www.ams.org/mathscinet-getitem?mr=#1}{#2}
}
\providecommand{\href}[2]{#2}

\end{document}